\documentclass[a4paper,11pt]{article}

\title{DLR equations and rigidity for the ${\rm Sine}$-beta process} 

\author{
\ David Dereudre\,\footnote{Univ. Lille, CNRS, UMR 8524 - Laboratoire Paul Painlev\'e, 
F-59000 Lille, France. \newline Email: \href{mailto:david.dereudre@math.univ-lille1.fr}{\nolinkurl{david.dereudre@univ-lille.fr}}},\quad
\ Adrien Hardy\,\footnote{Univ. Lille, CNRS, UMR 8524, Inria - Laboratoire Paul Painlev\'e, 
F-59000 Lille, France. \newline Email: \href{mailto:adrien.hardy@math.univ-lille1.fr}{\nolinkurl{adrien.hardy@univ-lille.fr}}},\quad
\ Thomas Lebl\'e\,\footnote{Courant Institute of Mathematical Sciences, 251 Mercer Street, New York University,
New York, NY 10012-1110, USA. Email: \href{mailto:thomasleble@gmail.com}{\nolinkurl{thomasleble@gmail.com}}}
,\quad
\ Myl\`ene Ma\"\i da\,\footnote{Univ. Lille, CNRS, UMR 8524 - Laboratoire Paul Painlev\'e, 
F-59000 Lille, France.  \newline Email: \href{mailto:mylene.maida@math.univ-lille1.fr}{\nolinkurl{mylene.maida@univ-lille.fr}}}
}

\usepackage{natbib}
\usepackage{anysize}
\marginsize{3cm}{3cm}{2.5cm}{2.5cm}
\usepackage[english]{babel}
\usepackage[T1]{fontenc}
\usepackage[utf8]{inputenc}
\usepackage{lmodern}
\usepackage{todo}
\usepackage{fancyhdr}
\usepackage{amssymb,amsmath,amscd,amsfonts,amsthm,bbm}
\usepackage{amsfonts}
\usepackage{amsmath}
\usepackage{mathtools}
\usepackage{mathrsfs} 
\usepackage{latexsym,lmodern,mathtools,color}
\usepackage{tikz,graphicx,subfigure,overpic}
\usepackage[colorlinks=true, linkcolor=blue, citecolor=blue]{hyperref}
\usepackage[fixlanguage]{babelbib}
\usepackage{color}
\usepackage{comment}
\usepackage{fancyhdr} 
\usepackage{natbib,lmodern,mathtools}
\usepackage{enumitem}
\setlist{nolistsep}

\usepackage{url}
\DeclareUrlCommand\email{\urlstyle{rm}}

\numberwithin{equation}{section}

\usepackage{amsthm}
 
\theoremstyle{definition}  
\newtheorem{theorem}{Theorem}[section]
\newtheorem{lemma}[theorem]{Lemma}
\newtheorem{corollary}[theorem]{Corollary}
\newtheorem{proposition}[theorem]{Proposition}

\newtheorem{definition}[theorem]{Definition}
\newtheorem{Remark}[theorem]{Remark}
\newenvironment{remark}{\begin{Remark}\rm}{\end{Remark}}

\newtheorem{Example}[theorem]{Example}

\newcommand{\eq}{\begin{equation}}
\newcommand{\qe}{\end{equation}}

\newcommand{\1}{\mathbf{1}}
\newcommand{\N}{\mathbb{N}}
\newcommand{\Z}{\mathbb{Z}}
\newcommand{\R}{\mathbb{R}}
\newcommand{\T}{\mathbb{T}}

\newcommand{\E}{\mathbb{E}}
\renewcommand{\P}{\mathbb{P}}

\DeclareMathOperator{\dist}{dist}

\newcommand{\bs}{\boldsymbol}
\newcommand{\bv}{\mathbf}

\def \Err{\mathrm{Err}}

\renewcommand{\epsilon}{ \varepsilon}
\renewcommand{\phi}{ \varphi}
\renewcommand{\d}{ {\rm d}}
\renewcommand{\frak}{\mathfrak}

\newcommand{\La}{\Lambda}
\newcommand{\leb}{\mathsf{Leb}}

\newcommand{\e}{\mathrm{e}}

\newcommand{\g}{\gamma}
\newcommand{\B}{\mathbf{B}}

\renewcommand{\div}{\mathrm{div}}

\newcommand{\sineb}{\mathrm{Sine}_{\beta}}
\newcommand{\Fluct}{\mathrm{Fluct}}
\newcommand{\Discr}{\mathrm{Discr}}
\newcommand{\Leb}{\mathsf{Leb}}
\newcommand{\Comp}{\mathsf{Comp}}
\newcommand{\V}{\mathsf{V}}
\newcommand{\Move}{\widetilde{\mathsf{M}}}

\newcommand{\II}{\mathsf{II}}
\newcommand{\Qnbeta}{Q_{\n, \beta}}
\newcommand \Gibbs{\textsf{Gibbs}}
\newcommand \Int{\textsf{Int}}
\newcommand \Ext{\textsf{Ext}}
\newcommand \Ener{\textsf{Energy}}
\newcommand \Intz{\textsf{Int}_0}
\newcommand \Moove{\textsf{Move}}
\newcommand \M{\mathsf{M}}
\newcommand{\per}{\textsf{per}}
\newcommand{\n}{n}
\newcommand{\A}{\mathsf{A}}
\renewcommand{\H}{\mathsf{H}}
\renewcommand{\L}{\mathsf{G}}

\newcommand{\Number}{\mathsf{Number}}
\newcommand{\CP}{\mathsf{C}_P}
\newcommand{\Conf}{\mathsf{Conf}}

\newcommand{\h}{\mathsf{Cost}}
\newcommand{\W}{\mathbb{W}}
\newcommand{\Q}{\mathsf{Q}}
\newcommand{\tCP}{\widetilde{\mathsf{C}}_P}
\newcommand{\Aa}{\mathsf{HasPoints}}

\newcommand{\NoRigid}{\mathsf{CanRmv}}

\newcommand{\create}{\mathsf{create}}
\newcommand{\Eve}{E}
\newcommand{\Dens}{\mathsf{Dens}}
\newcommand{\Test}{\mathsf{Test}}

\newcommand{\Nn}{\mathsf{N}}
\newcommand{\Term}{\mathsf{Term}}
\newcommand{\xn}{\bv x_n}

\makeatletter
\def\namedlabel#1#2{\begingroup
    #2%
    \def\@currentlabel{#2}%
    \phantomsection\label{#1}\endgroup
}
\makeatother
\begin{document}
\maketitle

\emph{
This paper is dedicated to the memory of our colleague Hans-Otto Georgii (1944-2017). He played a crucial role in the rigorous development of the theory of Gibbs measures.}
\\

\begin{abstract}
We investigate Sine$_\beta$, the universal point process arising as the thermodynamic limit of the microscopic scale behavior in the bulk of one-dimensional log-gases, or $\beta$-ensembles, at inverse temperature $\beta>0$. We adopt a statistical physics perspective, and give a description of Sine$_\beta$ using the Dobrushin-Lanford-Ruelle (DLR) formalism by proving that it satisfies the DLR equations: the restriction of Sine$_\beta$ to a compact set, conditionally on the exterior configuration, reads as a Gibbs measure given by a finite log-gas in a potential generated by the exterior configuration.   In short, $\sineb$ is a natural infinite Gibbs measure  at inverse temperature $\beta>0$ associated with the logarithmic pair potential interaction. Moreover, we show that Sine$_\beta$ is number-rigid and tolerant in the sense of Ghosh-Peres, i.e. the number, but not the position, of particles lying inside a compact set is a deterministic function of the exterior configuration. Our proof of the rigidity differs from the usual strategy and is robust enough to include more general long range interactions in arbitrary dimension. 
\end{abstract}

\section{Introduction and main results}
\subsection{The log-gas and the Sine-beta process}
The (finite, one-dimensional) log-gas is a random system of $\n$ particles confined on the real line $\R$, or the unit circle $\T$, interacting via a repulsive pair potential given by  the logarithm of the inverse distance between particles. Physically, it represents a statistical gas of identically charged particles living in a one-dimensional environment and interacting according to the laws of two-dimensional electrostatics. For a fixed value of the \textit{inverse temperature} parameter $\beta>0$, the distribution of the $\n$-tuple of particles is  given by the canonical Gibbs measure for this interaction, whose density reads as 
\begin{equation}
\label{loggasinformal}
\frac{\d \,\text{Log-gas}}{\d\, \text{Lebesgue}^{\otimes \n}} \propto \exp\left( - \beta \times \text{Logarithmic interaction energy of the particles} \right),
\end{equation}
where ``$\propto$'' means ``equal up to a multiplicative normalizing constant''.

\paragraph{Random matrices and $\beta$-ensembles.} 
From the statistical physics point of view, the log-gas is interesting because of the singular and long-range nature of its interaction potential.
Another important motivation comes from the link between log-gases and random matrix theory; we refer to \cite{forrester2010log} for an extensive treatment of this connection. There are several models of random $n\times n$ matrices whose random eigenvalues exhibit a joint density of the form,
\begin{multline}
\label{betaens}
\frac{\d \, \text{Eigenvalues}}{\d\, \text{Lebesgue}^{\otimes \n}} \propto \exp\left(- \beta  \sum_{j < l} - \log |x_j - x_l| \right) \prod_{j=1}^n\omega(x_j) 
\\ 
=  \prod_{j< \ell}^n \left|x_j-x_\ell\right|^{\beta}\prod_{j=1}^n\omega(x_j)\, ,
\end{multline}
where $\omega$ is an appropriate weight function supported on (a subset of) $\R$ or $\T$.
\begin{itemize}
\item [$\diamond$]The Gaussian ensembles correspond to $\omega(x)=\e^{-x^2}$. For the specific values $\beta=1$ or $2$ or $4$, one recovers the celebrated orthogonal/unitary/symplectic invariant random matrix ensembles. 
\item [$\diamond$] The Wishart, or Laguerre ensembles correspond to $\omega(x)=x^\alpha\e^{-x}\bs 1_{\R_+}(x)$.
\item [$\diamond$] The Jacobi ensembles (Manova) correspond to $\omega(x)=x^{\alpha_1}(1-x)^{\alpha_2}\bs 1_{[0,1]}(x)$.
\end{itemize}
These three models are random Hermitian matrices with eigenvalues on $\R$. Another example is given by the Circular $\beta$-ensemble, hereafter denoted C$\beta$E, which is a random unitary matrix with eigenvalues on the unit circle $\mathbb T$ and joint law  \eqref{betaens} with $\omega(x)=1$. 

When $\beta=2$, an integrable structure comes into play: the eigenvalues process is a determinantal point process, allowing exact computations for most quantities of interest, such as the correlation functions that one can express in terms of determinants. Similarly, when $\beta=1$ or $4$, some computations are still tractable due to a Pfaffian point process structure, although at the price of more involved formulas. In contrast, the present paper deals with “$\beta$ arbitrary”.

\paragraph{The $\sineb$ process.} 
Under an appropriate scaling within the bulk (i.e. the interior) of the spectrum, chosen so that the typical distance between consecutive points is of order $1$, the large $\n$ limit of the random point process of the eigenvalues exists and is called the $\sineb$ process.  Stated otherwise, $\sineb$ is the \textit{microscopic thermodynamic limit in the bulk} of one-dimensional log-gases, and it is the probability law of a certain random infinite point configuration on $\R$.

The $\sineb$ process is \textit{universal} in the sense it only depends on the inverse temperature $\beta>0$ and not on the initial weight $\omega$, for a large family of weights $\omega$, see \cite{BEY14,BEY12}. It is also “universal in the interaction”, in the sense that it appears when considering pair potential interactions that only have a logarithmic singularity at zero, see \cite{Ven13}. 

In the $\beta = 1, 2, 4$ cases, this limiting process is rather well understood due to the determinantal/Pfaffian structure, see e.g. \cite{DeiftGioev} and references therein. However, in the general $\beta>0$ setting, the mere existence of this limit is a difficult result, which was obtained, together with a rather involved description of the limiting object, in \cite{valko2009continuum} for the Gaussian $\beta$-ensemble and in \cite{killip2009eigenvalue} for the Circular $\beta$-ensemble.  The fact that these two descriptions coincide is checked e.g. in \cite{nakano2014level}. 

These descriptions of $\sineb$ use systems of coupled stochastic differential equations to derive the number of eigenvalues/particles falling in a given interval. This turns out to be tractable enough to study fine properties of the point process, such as to obtain large gap probability estimates \cite{ValkoVirag10}, a Central Limit Theorem \cite{kritchevski2012scaling} and large deviation and maximum deviation estimates \cite{holcomb2015large,holcomb2017overcrowding,HolcombePaquette18} for the number of points in an interval, as well as the Poissonian behavior of  $\sineb$  as $\beta \to 0$  \cite{allez2014sine}. 

More recently, the process has been characterized by \cite{valko2017sine} as the spectrum of a random infinite-dimensional operator. In particular this allows a better understanding on the $\beta$-dependency of the process \cite{ValkoVirag18}. 

However, among particularly relevant features of $\sineb$, its correlation functions, and even the asymptotic behavior of its two-point correlation function, remain unknown for  generic $\beta>0$; see however \cite[Chapter 13]{ForresterCbeta} for special cases. 

The goal of this work is to study  $\sineb$ from a statistical physics perspective, so as to obtain an alternative description as an infinite Gibbs measure, characterized by canonical Dobrushin-Lanford-Ruelle equations. In short, $\sineb$ is the natural infinite Gibbs measure  at inverse temperature $\beta>0$ associated with the logarithmic pair potential interaction. 

\subsection{Sine-beta  as an infinite Gibbs measure}
\subsubsection{Context for the DLR formalism}
For any fixed number $\n$ of particles, the canonical Gibbs measure of the Log-gas mentioned in \eqref{loggasinformal} minimizes the quantity
$$
\beta \times \text{Expected logarithmic energy} + \text{Entropy with respect to Lebesgue$^{\otimes \n}$}
$$
among all probability laws of random $\n$-point configuration. This is a famous variational principle for Gibbs measures, see e.g. \cite[Section 6.9]{friedli2017statistical} for a discussion. On the other hand, in the infinite-volume setting, it is shown in \cite[Corollary 1.2]{Leble:2017mz} that $\sineb$ minimizes a free energy functional of the type
\eq
\label{intuitiveFE}
\beta \times \text{Expected renormalized energy} + \text{Entropy with respect to Poisson},
\qe
among laws of stationary point processes, where $\mathrm{Poisson}$ is the Poisson point process with intensity 1 on $\R$. Here ``renormalized energy'' is a way to define the logarithmic energy of infinite configurations at microscopic scale, see Section \ref{sec:Energy}. It is thus natural to ask whether one can obtain a description of $\sineb$ as an infinite Gibbs measure. In view of \eqref{loggasinformal}, the naive guess would be that
\begin{equation} 
\label{naiveguess}
\frac{\d\,\sineb}{\d\,\text{Poisson}} \propto \exp\left(-\beta \times \text{Renormalized energy}\right),
\end{equation}
which is well-known to be illusory because any stationary process absolutely continuous with respect to a Poisson process is the said Poisson process itself.  The Dobrushin-Lanford-Ruelle (DLR) formalism provides the correct setting to possibly recast \eqref{naiveguess} in a \emph{local} way; we refer e.g. to the book  \cite{Georgiibook} for a general presentation in the lattice case, see  also \cite{dereudre2017introduction} for a pedagogical introduction in the setting of point processes. Informally, we will show that, given any bounded Borel set $\Lambda\subset\R$ and any configuration $\gamma_{\Lambda^c}$ outside of $\Lambda$, the law of the configuration $\gamma_{\Lambda}$ in $\Lambda$ drawn from $\sineb$ knowing the exterior configuration $\gamma_{\Lambda^c}$ can be written as
\begin{multline*}
\frac{\d\,\left( \sineb \text{ in $\Lambda \, \Big| \gamma_{\La^c}$}\right) }{\d\,\text{Lebesgue}^{\otimes N}}(\gamma_{\Lambda}) \\
\propto \exp\left(-\beta \times \text{Logarithmic energy of $\gamma_{\Lambda}$ with itself and with $\gamma_{\Lambda^c}$}\right),
\end{multline*}
where the number of points $N$ of the configuration $\gamma_\La$ is almost surely determined by the exterior configuration $\gamma_{\La^c}$. 

\subsubsection{Terminology and notation}
A point \emph{configuration} on $\R$ is a locally finite subset of $\R$ allowing multiple points. Formally, we identify a configuration $\g$ with an integer valued Radon measure on $\R$ that we still denote $\g$. That is we write
$\g = \sum_{x \in \g} \delta_x$,
and if $f$ is a test function, we let
$\int f \d \g  := \sum_{x \in \g} f(x).$ When $\g\in\Conf(\R)$ is such that $\g(\{x\})\in\{0,1\}$ for every $x\in\R$, we say that $\g$ is a \emph{simple} configuration. 
The space of point configurations $\Conf(\R)$, seen as a subspace of the Radon measures, is equipped with the topology coming by duality with the space of continuous functions $\R\to\R$ with compact support, making $\Conf(\R)$ a Polish space, see e.g. \cite[Section 15.7]{Kallenberg}. It is also the smallest topology that makes the mapping $\gamma\mapsto \gamma(B)$ continuous for every bounded Borel set $B$.
%for the maps  $\eta \mapsto \int f(x) d\eta(x)$, where $f$ is a continuous, compactly supported function. It is the topology induced by the topology of local convergence on the space of Radon measures. 
We then endow $\Conf(\R)$ with its Borel $\sigma$-algebra. In the following, by a \emph{point process} we mean a probability measure on simple configurations on $\R$, namely a probability measure on $\Conf(\R)$ such that  $\P(\gamma \text{ is simple})=1$. Let us introduce some additional notation: for any Borel set  $\Lambda\subset\R$,
\begin{itemize}
\item[$\diamond$] the set of all  configurations in $\La$ is denoted by $\Conf(\La)$
\item[$\diamond$] $\gamma\in\Conf(\R) \mapsto \gamma_\La\in\Conf(\La)$ stands for the restriction of a configuration to $\La$
%\item[$\diamond$]  $\Leb_{\Lambda}$ is the restriction of the Lebesgue measure to $\Lambda$,  $\Leb_{\Lambda}(\d x):=\bs 1_{\Lambda}(x)\d x$.
%\item[-] For all $n \in \N,$  we set $\Lambda_n := \left[- \frac{n}{2}, \frac{n}{2}\right]$, where $\N:=\{0,1,2,\ldots\}$
\item[$\diamond$]  $|\gamma|=\gamma(\R)\in\N\cup\{\infty\}$ stands for the number of points of a configuration $\gamma$.
%\item[$\diamond$]  A function $f:\Conf(\R)\to\R$ is \emph{local} if there exists a compact subset $K\subset\R$ such that $f(\gamma)=f(\gamma_K)$ for every $\gamma\in\Conf(\R)$. \\
\end{itemize}

\subsubsection{Statement of the results}
The central result of this paper is the next theorem.

\begin{theorem}
\label{theo:DLRmain} 
For any $\beta>0$ and any bounded Borel set $\Lambda\subset\R$ the following holds.
\begin{description}
\item[\namedlabel{Thm1A}{(A)}  Rigidity]
\label{theo:rigidity}
There exists a measurable function $\mathsf{Number}_\La: \Conf(\La^c)\to\N$ such that
$$
N:=|\gamma_\La|= \mathsf{Number}_\La(\gamma_{\La^c})\quad \text{ for } \sineb \text{-almost every (a.e.) }  \gamma.
$$

\item[\namedlabel{Thm1B}{(B)}  Definiteness of the exterior potential]
  For any $x\in\La$ and $\sineb$-a.e. configuration $\gamma$, the following limit exists and is positive,
\begin{equation}
\label{omegxknow}
\omega(x|\gamma_{\La^c}):=\lim_{p\to\infty} \prod_{\substack{u\in\gamma_{\La^c}\\|u|\leq p}}\left|1-\frac xu\right|^{\beta}.
\end{equation}
Moreover, the partition function
\begin{equation}
\label{def:ZgammaLa}
Z(\gamma_{\La^c}):=\int_{\La^N}  \prod_{j< \ell}^N \left|x_j-x_\ell\right|^{\beta}\prod_{j=1}^N\omega(x_j|\gamma_{\La^c}) \prod_{j=1}^N \d x_j
\end{equation}
is finite and positive.

\item[\namedlabel{Thm1C}{(C)}  DLR equations]
 For any bounded measurable function  $f:\Conf(\R)\to \R$,  we have
$$
\E_{\sineb}\big[f\big]
=\int \left[  \int_{\La^N} f(\{x_1,\ldots,x_N\}\cup\gamma_{\La^c})\; \rho_{\La^c}(x_1,\ldots,x_N)   \prod_{j=1}^N \d x_j\,\right]{\sineb}(\d\gamma)
$$
where, with the notation of \eqref{omegxknow}, \eqref{def:ZgammaLa}, we let $\rho_{\La^c}$ be defined by
\eq
\label{def:rho}
\rho_{\La^c}(x_1,\ldots,x_N):= \frac1{Z(\gamma_{\La^c})} \prod_{j< k}^N \left|x_j-x_k\right|^{\beta}\prod_{j=1}^N\omega(x_j|\gamma_{\La^c}).
\qe
\end{description}  
\end{theorem}

\paragraph{Number-rigidity:} 
Part \ref{Thm1A} of Theorem \ref{theo:DLRmain} states that, conditionally on the configuration outside a given bounded Borel set $\La$, the number of points drawn by $\sineb$ inside $\La$ is deterministic. This property has been recently put forward under the notion of \emph{number-rigidity} of point processes, starting from the pioneering work of \cite{ghosh2017rigidity}. Thus, $\sineb$ is number-rigid for any $\beta>0$.  The notion of rigid processes can be compared to the older notion of (non) \emph{hereditary} processes, see e.g. \cite[Definition 1]{dereudre2017introduction} for a presentation. Roughly speaking, if $P$ is not hereditary, it means that \textit{certain} points cannot be deleted with positive probability, whereas number-rigidity implies that \textit{no} point can be deleted. In particular, although $\sineb$ has finite specific relative entropy with respect to the Poisson point process, they are very different on this aspect - a Poisson process being, of course, far from rigid.

Shortly before the present work was completed, the rigidity property for $\sineb$ has been proven independently by \cite{chhaibi2018rigidity}.
Their proof follows the  strategy introduced by  \cite{ghosh2017rigidity}, namely to show that the variance of linear statistics for a smooth approximation of the characteristic function of a bounded interval can be made arbitrary small. To do so, they use variance estimates for polynomial test functions  that were proven in \cite{JiMa15} for the C$\beta$E, and proceed by approximation. The latter work relies on exact computations involving Jack's special functions which are tied to the specific structure of the Circular ensemble. 

In contrast, our proof for the number-rigidity only involves material from classical statistical physics and seems more flexible since it only relies on a weak form of DLR equations, the so-called \emph{canonical DLR equations}, and Campbell measures arguments. This may be of independent interest to prove number-rigidity for a larger class of point processes, in particular when the two-point correlation functions are not explicit, or simply not asymptotically tractable. To the best of our knowledge, this is the first alternative strategy with respect to the Ghosh-Peres method to prove number-rigidity. In particular, we  prove  the more general result that for large a class of long range interactions on $\R^d$, any solution of the canonical DLR equations is indeed number-rigid, see Theorem~\ref{DLRrigidityGeneral}.

\paragraph{The exterior potential:} The existence of the limit \eqref{omegxknow} is non-trivial and follows from quite subtle cancellations. In fact, it was expected in \cite[Section 12]{Ghosh-Lebowitz16} that it cannot be defined properly leading to the belief that the DLR equations were not reachable in this setting. 

\paragraph{Tolerance:}  There are other notions of rigidity than number-rigidity, such as \emph{barycenter-rigidity} or \emph{super-rigidity}. Barycenter-rigidity states that the barycenter of the configuration inside a domain is a deterministic function of the exterior configuration, and super-rigidity expresses the fact that the interior configuration is completely prescribed by the exterior. On the other hand, the notion of \emph{tolerance}, introduced in \cite{ghosh2017rigidity}, states that, roughly speaking, the number of points is the only rigid quantity prescribed by the exterior. It follows from part \ref{Thm1C} of Theorem \ref{theo:DLRmain} that $\sineb$ is tolerant, which is a new result for general $\beta>0$. More precisely, we have:

\begin{corollary}[$\sineb$ is tolerant] For any bounded Borel set $\La$ and $\sineb$-a.e. point configuration $\gamma$, the law of the particles drawn from $\sineb$ inside $\Lambda$ given the exterior configuration $\gamma_{\Lambda^c}$ is mutually absolutely continuous with respect to the $N$-fold Lebesgue measure, where $N$ is the number of points in $\Lambda$ (which is prescribed, by rigidity).
\end{corollary}

We also obtain the following result (which will follow from Corollary~\ref{DLRcharge}).
\begin{corollary}
\label{sinebcharge}
For any disjoint bounded Borel sets $B_1,\ldots,B_k\subset\R$ with positive Lebesgue measure and any integers $n_1,\ldots,n_k$,  we have
$$
\sineb\big(|\g_{B_1}|=n_1,\ldots,|\g_{B_k}|=n_k\big)>0.
$$
\end{corollary}

\paragraph{Relation to previous results:}  As mentioned above, when $\beta=2$ the process  benefits from an integrable structure with explicit correlations functions: it is the determinantal point process associated with the sine kernel. Using this structure, previous to the work of \cite{chhaibi2018rigidity},  \cite{Ghosh15} obtained the rigidity for $\beta=2$. Moreover, parts \ref{Thm1B} and \ref{Thm1C} of the theorem, and thus the tolerance, have been obtained by \cite{bufetov2016conditional} for a class of number-rigid determinantal point processes, including the Sine$_2$ process.

\subsection{Related questions and perspectives}
\paragraph{Fluctuations of smooth linear statistics.}
Let $\varphi:\R\to\R$ be a smooth and compactly supported test function. One may define the fluctuation of $\varphi$ as the random variable 
$$
\Fluct[\varphi](\gamma) := \int \varphi(x) \left( \gamma(\d x) - \d x \right),
$$
and ask for the behavior, as $\ell \to \infty$, of 
\begin{equation}
\label{fluctphil}
\Fluct[\varphi_{\ell}](\gamma) := \int \varphi\left(\frac{x}{\ell}\right) \left( \gamma(\d x) - \d x  \right).
\end{equation}
when the random configuration $\gamma$ has law $\sineb$. Having in mind similar results for $\beta$-ensembles, see e.g. \cite{shchange}, \cite{bekerman2016mesoscopic}, \cite{bekerman2017clt}, and since Theorem \ref{theo:DLRmain}\ref{Thm1C} shows that $\sineb$ is conditionally a $\beta$-ensemble, we could expect the fluctuation in \eqref{fluctphil} to converge in law \textit{without normalization} to a centered Gaussian random variable with standard deviation proportional to the fractional Sobolev $H^{1/2}$ norm of the test function $\varphi$. This Central Limit Theorem is proven in \cite{CLTfluct} for $\varphi$ smooth enough, using Theorem \ref{theo:DLRmain} as a key input.

\paragraph{Uniqueness.}
It is natural to ask the following:
\begin{itemize}
\item[(a)] Is $\sineb$ the only stationary process satisfying the DLR equations?
\end{itemize}
The answer to this question is positive when $\beta = 2$, as a consequence of the work \cite{kuijlaars2017universality} where, for Sine$_2$ a.e $\gamma\in\Conf(\R)$, the asymptotic of the conditional measure $\rho_{\La^c}(x_1,\ldots,x_N)$ has been shown to be Sine$_2$ in the limit where $\La:=[-R,R]$ and $R\to\infty$. We expect the answer to be positive for all values of $\beta$.
\begin{itemize}
\item[(b)] Is $\sineb$ the only minimizer of the free energy functional \eqref{intuitiveFE}? For a rigorous definition of this functional, see \cite{Leble:2017mz}.

\item[(c)] Are minimizers of the free energy the same as the solutions to the DLR equations?
\end{itemize}
 We expect both answers to be positive\footnote{See \cite{erbar2018one} for a positive answer to question (b).} for all $\beta>0$. Indeed, uniqueness of infinite-volume Gibbs measures is usually expected in dimension one, see e.g. \cite[Section 6.5.5]{friedli2017statistical} for such a result (that is not applicable here because our interaction is not short-range). For log-gases in dimension $2$, namely Coulomb gases, or in higher dimension, it might happen however that the uniqueness of minimizers/solutions to DLR equations, even up to symmetries, fails to hold for certain values of $\beta$.

\subsection{Strategy for the proof and plan of the paper}
In order to prove Theorem \ref{theo:DLRmain}, we first prove part \ref{Thm1B}, together with a weaker version of part \ref{Thm1C}, which form the \emph{canonical DLR equations}, where we further condition on the number $|\gamma_\La|$ of particles lying in $\La$. These two results form Theorem \ref{theo:DLRweak}, which is proven in Section~\ref{sec:DLR}. We start from the Circular ensemble, for which a, finite $n$, periodic version of the canonical DLR equations holds, and we perform several approximations. As a technical ingredient, we use \textit{discrepancy estimates}, i.e. controls on the difference
$|\gamma_{\La}| - |\La|$ between the number of points of a typical configuration $\gamma$ in a bounded set $\La$, and the size of $\La$ (in the sense of its Lebesgue measure).

Next, in Section~\ref{sec:rigidity}, we leverage Theorem \ref{theo:DLRweak} in order to obtain the rigidity result of Theorem \ref{theo:DLRmain} part \ref{Thm1A}. In fact, we prove a possibly more general result: Any stationary point process $P$ satisfying the canonical DLR equations is number-rigid, see Theorem~\ref{Theo-rigidity}. The proof of rigidity goes by contradiction: if $P$ were not rigid, then we could deduce from the canonical DLR equations a particular structure for its Campbell measures that turns out to be absurd due to the long range nature of the logarithmic interaction. This result is stated and proven for the logarithmic interaction and in dimension one, but our proof is robust enough to yield that the  same result,  that we state in Theorem~\ref{DLRrigidityGeneral}, holds for more general long range interactions in dimension $d\geq 1$.

%\paragraph{Acknowledgments:} This work has been partially supported by  ANR JCJC {BoB} (ANR-16-CE23-0003) and Labex {CEMPI} (ANR-11-LABX-0007-01). 

\section{Canonical DLR equations}
\label{sec:DLR}
We consider here and prove a weaker version of Theorem \ref{theo:DLRmain}, which we refer to as the \emph{canonical DLR equations}, which involves conditioning on the number of particles lying inside $\La$. The term \textit{canonical} refers to the fact that the number of particles is fixed, as in the \textit{canonical} ensemble of statistical physics or the \textit{canonical} Gibbs measure, in contrast with e.g. a \textit{grand canonical} setting.

\begin{theorem}
\label{theo:DLRweak} 
For any $\beta>0$ and any bounded Borel set $\Lambda\subset\R$ the following holds true.
\begin{description}
\item[\namedlabel{Thm2Bbis}{(B)}  Definiteness of the exterior potential]
For any $x\in\La$ and $\sineb$-a.e.  $\gamma$, the following limit exists and is positive,
\eq
\label{DLRweight}
\omega(x|\gamma_{\La^c}):=\lim_{p\to\infty} \prod_{\substack{u\in\gamma_{\La^c}\\|u|\leq p}}\left|1-\frac xu\right|^{\beta}.
\qe
Moreover, for $\sineb$-a.e.  $\gamma$,  the partition function %\MM{j'ai remplacé $N$ par $|\gamma_\La|$}
\eq
\label{DLRZ}
Z(\gamma_{\La^c},|\gamma_\La|):=\int_{\La^{|\gamma_\La|}}\prod_{j< \ell}^{|\gamma_\La|} \left|x_j-x_\ell\right|^{\beta}\prod_{j=1}^{|\gamma_\La|}\omega(x_j|\gamma_{\La^c}) \d x_j
\qe
is finite and positive. 
\item[\namedlabel{Thm2Cbis}{(C*)}  Canonical DLR equations]
Let $f : \Conf(\R) \to \R$ be a bounded, measurable function. We have the identity
\begin{multline*}
\E_{\sineb}\big[f\big] =\int \left[  \int f(\{x_1,\ldots,x_{|\gamma_\La|}\}\cup\gamma_{\La^c})\;\rho_{\La^c}(x_1,\ldots,x_{|\gamma_\La|}) \;\prod_{j=1}^{|\gamma_\La|} \d x_j \,\right]{\sineb}(\d\gamma),
\end{multline*}
with $\rho_{\La^c}$ defined similarly to \eqref{def:rho} by 
\eq \label{def:rhoweak}
\rho_{\La^c}(x_1,\ldots,x_{|\gamma_\La|}):= \frac1{Z(\gamma_{\La^c},{|\gamma_\La|})} \prod_{j< k}^{|\gamma_\La|} \left|x_j-x_k\right|^{\beta}\prod_{j=1}^{|\gamma_\La|}\omega(x_j|\gamma_{\La^c}).
\qe
\end{description}  
\end{theorem}
Parts \ref{Thm2Bbis} of Theorem \ref{theo:DLRweak} and Theorem \ref{theo:DLRmain} are the same statement. It will be thus enough to combine Part \ref{Thm1A} of Theorem~\ref{theo:DLRmain} with Theorem \ref{theo:DLRweak} in order to obtain Part \ref{Thm1C} of Theorem \ref{theo:DLRmain} and to conclude the proof - this shall be done in Section \ref{sec:rigidity} and will strongly rely on the \textit{canonical} form of the DLR equations. 

The remainder of this section is devoted to the proof of Theorem~\ref{theo:DLRweak}. 

\paragraph{Some terminology.} 
\begin{itemize}
\item[$\diamond$] We say that a measurable function $f$ is \textit{local} if
there exists a compact subset $K\subset\R$ such that $f(\gamma)=f(\gamma_K)$ for every $\gamma\in\Conf(\R)$. 
\item[$\diamond$] We use the notation $\La_p$ for the line segment $[-\frac p2,\frac p2]$. 
\item[$\diamond$] If $n$ is an integer, and $\La$ a bounded Borel set, we denote by $ \bv B_{n, \La}(\d \eta)$ the law of the Bernoulli process drawing a random configuration of $n$ independent points  uniformly in $\La$.  
%\item[$\diamond$] \comT{We say that a statement holds “when $1 \ll p \ll n$” if it holds for $p$ large enough, and for $n \geq p$ large enough (possibly depending on $p$).}
\end{itemize}

For the rest of this section,  we fix $\Lambda\subset\R$ a bounded Borel set and $f:\Conf(\R)\to\R$ a bounded Borel \emph{local} function.  The integer $p$ used below is always assumed large enough so that $\La\subset\La_p$.

\subsection{Move functions, Gibbs kernels, and proof of Theorem \ref{theo:DLRweak}\ref{Thm2Bbis}}

\label{sec:movingpointsfinite}
First, we introduce \emph{move functions} and \emph{Gibbs kernels}, and use them to rephrase Theorem~\ref{theo:DLRweak}.

\subsubsection{Heuristics for the move functions}
The basic idea is the following: for a given bounded Borel set $\Lambda$ we want to define a Gibbs measure, formally written $\Gibbs_\La(\,\cdot\, | \Ext)$ on interior configurations $\Int$ in $\La$, given an exterior configuration $\Ext$ in $\La^c$, namely
$$
\Gibbs_\La(\Int | \Ext) \propto \exp \left(-\beta \,\Ener\left(\Int \cup \Ext\right)\right).
$$
The physical energy reads
$$
\Ener\left(\Int \cup \Ext\right) = \Int \times \Int +  \Ext \times \Ext +  2 \cdot \Ext \times \Int, 
$$
where the product sign means “interact with”, but the exterior configuration $\Ext$ could be infinite, yielding two infinite terms in the definition of the energy. Since $\Ext$ is fixed, the interaction term $\Ext \times \Ext$ can be absorbed by the normalization constant, and we are left with
\begin{equation}
\label{gibbsintext}
\Gibbs_\La(\Int | \Ext) \propto \exp \left(-\beta\left( \Int \times \Int +  2 \cdot\Ext \times \Int \right) \right).
\end{equation}
To deal with the fact that $\Int \times \Ext$ could correspond to an infinite, or undefined, summation, we introduce the move functions: we fix some reference interior configuration $\Intz$ and we define $\Moove_\La(\Int, \Ext)$ as the energetic cost to move the points in $\La$ from $\Intz$ to those of $\Int$, as felt by the points of $\Ext$, namely:
\begin{equation}
\label{MooveA}
\Moove_\La(\Int, \Ext) := 2 \cdot \Ext \times \Int - 2 \cdot\Ext \times \Intz = 2 \cdot\Ext \times (\Int - \Intz).
\end{equation}
Thus we can write,
\begin{equation}
\label{extintmoove}
2 \cdot\Ext \times \Int = \Moove_\La(\Int, \Ext) + 2 \cdot\Ext \times \Intz.
\end{equation}
The second term in the right-hand side of \eqref{extintmoove} is independent of $\Int$ and can again be absorbed by the normalization constant in \eqref{gibbsintext}, yielding
$$
\Gibbs_\La(\Int | \Ext) \propto  \exp \left(-\beta \left(\Int \times \Int +  \Moove_\La(\Int, \Ext) \right) \right),
$$
which is the form of the \textit{Gibbs kernels} (or \textit{Gibbs specifications}) that we introduce more precisely below. If $\Ext$ is infinite, it is still not clear why $\Moove_\La(\Int, \Ext)$ would yield a finite quantity, but we can hope for some compensation between the two terms $\Ext \times \Int$ and $\Ext \times \Intz$. These  manipulations are valid in the finite case, and in the case of an infinite exterior configuration the rigorous approach consists in truncating the said configuration outside some large, but finite domain, and to show that the infinite-volume limit of these partial move functions exists, as we shall do next.

\subsubsection{Gibbs kernels (finite window)}
\label{sec:DLRforsinebetaprocess}
%In order to prove Part \ref{Thm2Bbis} of Theorem \ref{theo:DLRweak}, we wil need to define the move functions for slightly more general objects than \emph{configurations}. We say that a Radon measure $\eta$ on $\R$ is a \emph{generalized configuration} and we write $\eta \in \textsf{Gconf}(\R),$ if there exist an integer $n$ and $(x_1, \ldots, x_n) \in \R^n$ such that $\eta = \sum_{i=1}^n \delta_{x_i}.$ We denote by $|\eta| =n$ the number of points of $\eta$ counted with multiplicity. Otherwise stated, in this particular context, we want to allow multiple points.

\begin{definition}[Logarithmic interaction]
\label{defi:nonperiodicDLR}
We set for convenience\footnote{This allows to include the diagonal in double sums, since $g(x_i - x_i)$ is set to be zero. Of course, $g$ is not continuous at $0$, but neither is $\log$...}
\eq
\label{def:g}
g(x):=
\begin{cases}
-\log|x|& \text{ if } x\in\R\setminus\{0\}\\
0 & \text{ if } x=0
\end{cases}\qe

For any $p\geq 1,$ any $\gamma,\eta \in\Conf(\R)$, we introduce\footnote{In these definitions, the $1/2$ factor has been introduced for aesthetic reasons, so that the usual C$\beta$E  is an approximation of  $\sineb.$}:
 \begin{align}
\label{def:Hlambda}  \text{(Interaction energy)} &&\H_\Lambda(\g) & := \frac{1}{2}\iint g(x-y)\, \d \g_\La ^{\otimes 2} 
\\
\label{def:Mlambda}  \text{(Move function)} && \M_{\La, \La_p}(\eta, \g) & :=  \iint g(x-y) \,\d (\eta_\La-\g_\La) \otimes \g_{\La_p \setminus \La} 
\\
  \text{(Normalization)} &&Z_{\La, \La_p}( \g) &:= \int \e^{- \beta (\H_\Lambda(\eta)+  \M_{\La, \La_p}(\eta, \g))} \, \bv B_{|\g_\La|, \La}(\d \eta)
  \\
 \label{def:Llalap} \text{(Gibbs kernel)} &&\L_{\La, \La_p}(\d \eta, \g) &:= \frac{1}{ Z_{\La, \La_p}( \g)} \,\e^{- \beta(\H_\Lambda(\eta)+ \M_{\La, \La_p}(\eta, \g))} \, \bv B_{|\g_\La|, \La}(\d \eta).
 \end{align}

The measure $\L_{\La, \La_p}(\d \eta, \g)$ is the central object for the DLR formalism, called a \textit{Gibbs kernel}, or \textit{specification}.  Moreover we set,
 \begin{equation}
 \label{fLaR}
 f_{\La, \La_p}(\gamma): = \int f(\eta \cup \gamma_{\La_p\setminus\La}) \,\L_{\La, \La_p}(\d\eta, \g).
 \end{equation}
  \end{definition}
  
The right-hand side of \eqref{fLaR} should be read as follows: take $\gamma\in\Conf(\R)$ and a test function~$f$. Keep the configuration $\gamma$ in $\Lambda_p \backslash \Lambda$ and sample a new set of points $\eta$ inside $\Lambda$ with the same number of points as the old one, namely $|\gamma_{\Lambda}|$. This sampling is done according to the Gibbs measure $\L_{\La, \La_p}(\d \eta, \g)$ defined in \eqref{def:Llalap}, associated to the inverse temperature $\beta$ and the energy $\H_\Lambda(\eta)+  \M_{\La, \La_p}(\eta, \g)$, which represents the sum of the interaction energy of the points of $\eta$ (inside $\Lambda$) with themselves and the cost of moving points from the old configuration $\gamma_\Lambda$ to the new one $\eta$, as felt by the points  in $\Lambda_p \backslash \Lambda$ (in particular, the “reference interior configuration” $\Intz$ as in \eqref{MooveA} is here chosen to be $\gamma_{\La}$). Then, combine $\eta$ and $\gamma_{\Lambda_p \backslash \Lambda}$ into a configuration in $\Lambda_p$, and test it against the function $f$. The quantity $f_{\La, \La_p}(\g)$ is the expectation of this operation; let us emphasize that the randomness comes from the re-sampling in $\La$, while $\gamma_{\La^c}$ is fixed.

  \subsubsection{Gibbs kernels (infinite window) and reformulation of Theorem \ref{theo:DLRweak}}
We will show in Lemma \ref{lem:definimoveinf} below that the limit $p\to\infty$ of the quantities in Definition \ref{defi:nonperiodicDLR} exists. In particular, we obtain:

  \begin{lemma}[Existence of the move functions]
  \label{existsineb}
  For $\sineb$-a.e. configuration $\gamma$ and any $\eta\in\Conf(\La)$ satisfying $|\g_\La| = |\eta|,$ the limit
  \eq
  \label{convMove}
  \M_{\La,\R}(\eta, \g) :=\lim_{p\to\infty}  \M_{\La,\La_p}(\eta, \g)
  \qe
  exists and is finite. Moreover, for $\gamma$ fixed, the convergence in \eqref{convMove} is uniform in the choice of $\eta$.
  \end{lemma}
  This allows us to extend the notation of Definition \ref{defi:nonperiodicDLR} to “$p = +\infty$”, namely to give a meaning to 
  \begin{align}
  Z_{\La, \R} &:= \int \e^{- \beta(\H_\Lambda(\eta)+  \M_{\La, \R}(\eta, \g))} \, \bv B_{|\g_\La|, \La}(\d \eta).\\
\L_{\La, \R}(\d \eta, \g) &:= \frac{1}{ Z_{\La, \R}( \g)} \,\e^{- \beta(\H_\Lambda(\eta)+  \M_{\La, \R}(\eta, \g))} \, \bv B_{|\g_\La|, \La}(\d \eta).
\\
\label{fKernel}
 f_{\La, \R}(\gamma)& : = \int f(\eta \cup \gamma_{\La^c}) \,\L_{\La, \R}(\d\eta, \g).
 \end{align}
 \begin{remark}
 \label{movedontmove}
 At this point, an important observation is that the probability measure $\L_{\La, \R}(\d \eta, \g)$ only depends on the exterior $\gamma_{\La^c}$ and the number of points $|\gamma_{\La}|$. Indeed, if one changes $\gamma_\La$ to another configuration with the same number of points, then the move functions is changed by an additive constant which can be incorporated into a new partition function.
 \end{remark}
 Moreover, Lemma \ref{existsineb} already yields part \ref{Thm2Bbis} of the theorem.  
  
  \begin{proof}[Proof of Theorem \ref{theo:DLRweak}, part \ref{Thm2Bbis}] Since $\sineb$ is stationary, we know that $0$ is not a point of $\gamma$ for $\sineb$-a.e $\gamma$. Moreover, since $|\gamma_\La|=|\eta|$, one can write
  $$
  \M_{\La,\La_p}(\eta, \g)=\sum_{\substack{u\in\gamma_{\La^c}\\|u|\leq p}} \left(\sum_{x\in\gamma_\La}-\sum_{x\in\eta}\right)\log \left|1-\frac xu\right|,
  $$
  where the sum over the points of $\eta$ is to be understood with multiplicity.
  This yields
  \eq
  \label{linkGibbsbetaens}
  \e^{-\beta \M_{\La,\La_p}(\eta, \g)}= \frac{\prod_{x\in\eta}\Big(\prod_{\substack{u\in\gamma_{\La^c}\\|u|\leq p}}  \left|1-\frac xu\right|^{\beta}\Big)}{\prod_{y\in\gamma}\Big(\prod_{\substack{u\in\gamma_{\La^c}\\|u|\leq p}}  \left|1-\frac yu\right|^{\beta}\Big)} = \frac{\prod_{x\in\eta}\Big(\prod_{\substack{u\in\gamma_{\La^c}\\|u|\leq p}}  \left|1-\frac xu\right|^{\beta}\Big)}{\e^{-\beta \M_{\La,\La_p}( |\gamma_\La| \delta_0, \g)}},
  \qe
Letting $p\to\infty$, and using Lemma~\ref{existsineb} we obtain the existence of the weight $\omega(\cdot|\gamma_{\La^c})$, as defined in \eqref{DLRweight}.
   
  Clearly, $Z_{\La,\La_p}(\gamma) \in (0,+\infty)$ for every $p$, and the convergence \eqref{convMove} is uniform in $\eta$, hence $Z_{\La,\R}(\gamma) \in (0,+\infty)$. We may observe that, by definition of $Z(\gamma_{\La^c}, |\gamma_\La|)$, as in \eqref{DLRZ}, and by \eqref{linkGibbsbetaens}, we have $Z_{\La,\R}(\gamma) =Z(\gamma_{\La^c},|\gamma_\La|)),$ which concludes the proof.
  \end{proof}

Finally, let us observe that part \ref{Thm2Cbis} of Theorem \ref{theo:DLRweak} can be written as
\begin{equation}
\label{canonicalDLR}
\E_{\sineb}(f-f_{\La, \R}) = 0.
\end{equation}
The remainder of this section is devoted to prove \eqref{canonicalDLR}. To do so, we rely on the convergence of the microscopic statistics of C$\beta$E towards $\sineb$.

\subsection{Circular ensembles and the associated DLR equations}
\label{sec:finitemodel}
\subsubsection{Circular ensembles as model of log-gases}
As mentioned in the introduction, the one-dimensional log-gas is related to the \textit{Circular ensembles} appearing in random matrix theory. In the C$\beta$E, see e.g. \cite{forrester2010log}, the joint law of the eigenvalues $(\e^{i\theta_1}, \dots, \e^{i\theta_n})$ is given by 
\eq
\label{CbetaElaw}
\frac{1}{C_{n, \beta}} \prod_{j <\ell} \left|\e^{i \theta_j}- \e^{i \theta_\ell}\right|^{\beta} \prod_{j=1}^n \d\theta_j,
\qe
where $\d \theta_j$ is the Lebesgue measure on $[-\pi, \pi]$. The normalization constant $C_{n,\beta}$ is known: 
\eq
\label{partitionCbetaE}
C_{n, \beta} := (2\pi)^n \frac{\Gamma(\frac\beta 2 n+1)}{\Gamma(\frac\beta 2+1)^n}.
\qe
In this section, the central role is played by the point process $Q_{\n, \beta}$, defined as the push-forward of the density in \eqref{CbetaElaw} by the map
$$
(\theta_1, \dots, \theta_n) \mapsto \left(\frac{n \theta_1}{2\pi}, \dots, \frac{n\theta_n}{2\pi}\right).
$$
The realizations of $Q_{n, \beta}$ are point configurations in $\La_n = [-\frac n 2,\frac n2]$, that we will use as finite window approximations of $\sineb$. It can be equivalently defined as follows.

\begin{definition}[Periodic logarithmic interaction]
\label{sec:defi_finite_model}
For any $\n\geq 1$, we set
\eq
\label{def:gng}
g_n(x):=
\begin{cases}
 - \log\left|2 \sin \left(\displaystyle\frac{\pi x}{n}\right)\right|& \text{ if } x\in\R\setminus\{0\}\\
\;0 & \text{ if } x=0
\end{cases}.
\qe
For any  $\gamma \in\Conf(\R),$  we introduce the interaction energy
 \begin{equation}
\label{def:HnLambda} \H^{n-\per}_\Lambda(\g) := \frac{1}{2}\iint g_n(x-y) \d \g_\La^{\otimes 2}.
\end{equation}
\end{definition}

\begin{definition}[The canonical Gibbs measure of the finite periodic log-gas]
For any $\beta>0$, we consider the point process of $n$ points in $\Lambda_n$ given by
\eq
\label{eq:Qper}
\Qnbeta(\d \gamma):=\frac1{Z_{n, \beta}}\ \e^{-\beta \,\H^{n-\per}_{\Lambda_n}(\g) }\bv B_{n, \La_n}(\d \gamma),
\qe
where the partition function is given by, see \eqref{partitionCbetaE},
\eq
\label{Partitionbeta}
 Z_{n, \beta}:=\int \e^{-\beta \H^{n-\per}_{\Lambda_n}(\g) }\bv B_{n, \La_n}(\d \gamma)= \frac{C_{n, \beta}}{(2\pi)^n }= \frac{\Gamma(\frac\beta 2 n+1)}{\Gamma(\frac \beta 2+1)^n}.
\qe
\end{definition}
The probability measure $Q_{n, \beta}$ can be taken as a model of a finite log-gas:  \eqref{eq:Qper} gives a rigorous meaning to the informal definition \eqref{loggasinformal}. Let us emphasize that the choice of a \textit{periodic} logarithmic potential is not the usual one, but it is more convenient for us, and yields the same microscopic limit.

\subsubsection{Gibbs kernels (periodic setting)}
We introduce the following notation, which should be compared to Definition \ref{defi:nonperiodicDLR}. We add the superscript ${n-\per}$, in order to stress the fact that the $n$-periodic logarithmic interaction is used.
\begin{definition} For any $\g,\eta\in\Conf(\R)$,
\begin{align}
\label{def:Mnlambda}  \text{(Move function)} && \M_{\La, \La_p}^{n-\per}(\eta, \g) & :=  \iint g_n(x-y) \,\d (\eta_\La-\g_\La) \otimes \g_{\La_p \setminus \La} 
\\
 \label{def:Zper}  \text{(Normalization)} &&Z_{\La, \La_p}^{n-\per}( \g) &:= \int \e^{- \beta (\H^{n-\per}_\Lambda(\eta)+  \M^{n-\per}_{\La, \La_p}(\eta, \g))} \, \bv B_{|\g_\La|, \La}(\d \eta)
  \\
 \label{def:Llalapn} \text{(Gibbs kernel)} &&\L_{\La, \La_p}^{n-\per}(\d \eta, \g) &:= \frac{1}{ Z^{n-\per}_{\La, \La_p}( \g)} \,\e^{- \beta(\H^{n-\per}_\Lambda(\eta)+  \M^{n-\per}_{\La, \La_p}(\eta, \g))} \, \bv B_{|\g_\La|, \La}(\d \eta).
 \end{align}
We moreover use the notation,
\eq
 \label{fLaRn}
f^{n-\per}_{\La, \La_p}(\gamma)  : = \int f(\eta \cup \gamma_{\La_p\setminus\La}) \,\L^{n-\per}_{\La, \La_p}(\d\eta, \g).
\qe
\end{definition}

\subsubsection{DLR equations for the Circular ensemble}
The first step towards the canonical DLR equations \eqref{canonicalDLR} is  the following result.
\begin{proposition}[Canonical DLR equations for the finite log-gas] 
\label{nperiodicDLR}
For any bounded Borel set $\Lambda \subset \Lambda_n$ and any bounded Borel function $f:\Conf(\La_n)\to\R,$ we have
\begin{equation}
\label{DLRfinitemodel}
\E_{\Qnbeta}(f -  f^{n-\per}_{\La, \La_n}) = 0.
\end{equation}
\end{proposition}

\begin{proof}
Let us write for convenience $\La^c:= \La_n\setminus \La$.  The definitions   \eqref{def:Llalapn}--\eqref{eq:Qper}  of $f^{n-\per}_{\La, \La_n}$ and $\Qnbeta$ yield
\begin{multline*}
Z_{n,\beta} \,\E_{Q_{n, \beta}}(f^{n-\per}_{\La, \La_n}) = \\  \iint  f(\eta \cup \g_{\La^c}) \frac{1}{ Z^{n-\per}_{\La, \La_n}( \g) } \e^{-\beta ( \H^{n-\per}_\Lambda(\eta)+  \M^{n-\per}_{\La, \La_n}(\eta, \g))}
 \e^{-\beta \H^{n-\per}_{\La_n}(\g)} \bv B_{|\g_\La|, \La}(\d \eta) \bv B_{n, \La_n}(\d \g).
\end{multline*}
Next, we use  the following algebraic identity: for any $\eta\in\Conf(\La)$,
\begin{equation}
\label{algebraicidentity}
\H^{n-\per}_{\La_n}(\gamma) + \M^{n-\per}_{\Lambda,\Lambda_n}(\eta, \gamma)+ \H_{\Lambda}^{n-\per}(\eta)=\H^{n-\per}_{\La_n}(\eta \cup \gamma_{\Lambda^c})+\H^{n-\per}_{\Lambda}(\gamma).
\end{equation}
It can be easily checked from the definitions; indeed, using the informal notation ``$\times$'' for ``interact with'', we can write
\begin{align*}
(\eta \cup \gamma_{\La^c})^{\times 2} & = \eta^{\times 2} + \gamma_{\La^c}^{\times 2} + 2 \cdot \eta \times \gamma_{\La^c} \\
\gamma^{\times 2} = (\gamma_{\Lambda} \cup \gamma_{\Lambda^c})^{\times 2}  & = \gamma_{\Lambda}^{\times 2} + \gamma_{\La^c}^{\times 2} + 2 \cdot \gamma_{\La} \times \gamma_{\La^c},
\end{align*}
and substracting the second line from the first one, we obtain \eqref{algebraicidentity}. Combined with the relation,
$$
Z_{\La,\La_n}^{n-\per}(\gamma)=Z_{\La,\La_n}^{n-\per}(\eta\cup\gamma_{\La^c})\,\e^{-\beta \M_{\La,\La_n}(\eta,\gamma)}=Z_{\La,\La_n}^{n-\per}(\eta\cup\gamma_{\La^c})\,\e^{\beta \M_{\La,\La_n}(\gamma,\eta\,\cup\gamma_{\La^c})},
$$
this yields
\begin{multline*}
Z_{n,\beta} \,\E_{Q_{n, \beta}}(f^{n-\per}_{\La, \La_n}) = \iint  f(\eta \cup \g_{\La^c}) \frac{1}{ Z^{n-\per}_{\La, \La_n}( \eta\cup\g_{\La^c}) }\\
\times \e^{-\beta (\H^{n-\per}_{\La_n}(\eta \,\cup \gamma_{\Lambda^c})+\M_{\La,\La_n}(\gamma,\eta\,\cup\gamma_{\La^c})+\H^{n-\per}_{\Lambda}(\gamma))}
 \bv B_{|\g_\La|, \La}(\d \eta) \bv B_{n, \La_n}(\d \g).
\end{multline*}
Making the change of variables $(\eta,\gamma)\mapsto(\zeta,\xi)$ with $\xi:=\eta\cup\gamma_{\La^c}$ and $\zeta:=\gamma_\La$, we obtain
\begin{align*}
&Z_{n,\beta} \,\E_{Q_{n, \beta}}(f^{n-\per}_{\La, \La_n})\\ 
&= \iint  f(\xi) \frac{1}{ Z^{n-\per}_{\La, \La_n}( \xi )} \e^{-\beta (\H^{n-\per}_{\La_n}(\xi)+\M_{\La,\La_n}(\zeta,\xi)+\H^{n-\per}_{\Lambda}(\zeta))}
 \bv B_{|\xi_\La|, \La}(\d \zeta) \bv B_{n, \La_n}(\d \xi)\\
& = Z_{n,\beta} \,\E_{Q_{n, \beta}}(f),
\end{align*}
and the proposition is proved.
\end{proof}

\subsubsection{Local convergence of the circular ensembles to Sine-beta}
\label{sec:DLRequations}
As mentioned in the introduction, the point process $Q_{n,\beta}$ is known to converge towards $\sineb$ as $n\to\infty$. 
This convergence was first studied in \cite{killip2009eigenvalue}. They have shown that for any
$C^\infty$ function $\psi : \R \to \R$ with compact support,
\eq \label{convlaplace}
\lim_{n\to\infty}\E_{Q_{n,\beta}}\left[\exp\left(\int \psi\,\d\gamma\right)\right]=\E_{\sineb}\left[\exp\left(\int \psi\,\d\gamma\right)\right].
\qe
To be more precise, the limiting process was called $C_{\beta}E$ in \cite{killip2009eigenvalue}; almost at the same time, \cite{valko2009continuum} have shown
a similar result for the bulk limit of the Gaussian $\beta$-ensemble, with a limiting process called $\sineb.$
A bit later, it has been observed that $\sineb$ and $C_{\beta}E$ are identical, so that the result from 
\cite{killip2009eigenvalue} can now be stated under the form \eqref{convlaplace}; see \cite[Corollary 1.7]{nakano2014level} or \cite[Theorem 28]{valko2017sine}.  
For our purpose, we will need a slightly stronger convergence result that we state in the next proposition.

\begin{proposition}
\label{prop:sineb}
The sequence of point processes  $(Q_{n,\beta})_{n \geq 1}$ converges to the point process $\sineb$ in the topology of local convergence: for any bounded, Borel and local test function $\phi:\Conf(\R)\to\R$, we have
$$
\lim_{n\to\infty}\E_{Q_{n,\beta}}[\phi]=\E_{\sineb}[\phi].
$$
\end{proposition}

\begin{proof} The convergence \eqref{convlaplace} implies the weak convergence   $Q_{n,\beta}\to\sineb,$ see \cite[Definition 1.2 and 1.3]{killip2009eigenvalue} and \cite[Proposition 11.1.VIII]{Dave08}.

Now, we show that the sequence   $(Q_{n,\beta})_{n \geq 1}$ has an accumulation point in the local topology. Indeed, consider the relative entropy of two point processes $P,Q$ defined by
 $$
 I(P|Q):=\int \log \frac{\d P}{\d Q} \,\d P
 $$ 
 when $P$ has a density with respect to $Q$ and set $I(P|Q):=+\infty$ otherwise. Let $\Pi_{\La_n}$ be  the Poisson point process of intensity $1$ on $\La_n.$ According to
 \cite[Proposition 2.6]{GZ93},   it is enough to check that
 \eq
 \label{enttight}
 \sup_{n \in \N^*} \frac{1}{n } I(Q_{n,\beta} | \Pi_{\La_n}) <\infty.
 \qe
\begin{eqnarray*}
  I(Q_{n,\beta} | \Pi_{\La_n})&  =  & \int \log \frac{\d Q_{n,\beta}}{\d \bv B_{n, \La_n}} \d Q_{n,\beta} + \int  \log \frac{\d \bv B_{n, \La_n}}{\d \Pi_{\La_n}} \d Q_{n,\beta} \\ 
  & = & - \log Z_{n,\beta} - \beta \,\E_{Q_{n,\beta}} \Big[\H^{n-\per}_{\La_n}(\gamma)\Big] - \log \left(\e^{-n}\frac{n^n}{n!}\right).
\end{eqnarray*}
Recalling \eqref{Partitionbeta} and using Stirling formula, we see that
\eq \label{Znlimit}
\lim_{n \rightarrow \infty} \frac{1}{n} \log Z_{n, \beta} -\frac\beta 2 \log \frac n{2\pi}
\qe
exists and is finite.
Moreover,
\begin{align*}
\E_{Q_{n,\beta}} \Big[\H^{n-\per}_{\La_n}(\gamma)\Big]  & =  \frac{1}{Z_{n, \beta}} \int \H_{\La_n}^{n-\per}(\gamma) \e^{-\beta \H_{\La_n}^{n-\per}(\gamma)}\bv B_{n, \La_n}(\d \gamma) \\
& = - \frac{\d}{\d\beta} \log Z_{n, \beta} = -\frac{\d}{\d\beta}\log\frac{\Gamma(\frac\beta 2 n+1)}{\Gamma(\frac \beta 2+1)^n}
\end{align*}
which yields that
\begin{multline*}
\frac 1 n \E_{Q_{n, \beta}}\left[\H_{\Lambda_n}^{n-\per}(\g)\right] + \frac{1}{2} \log\frac{n}{2\pi}= \\
- \frac{1}{2} \psi\left(\frac \beta 2 n+1\right) +  \frac{1}{2} \psi\left(\frac\beta 2+1\right) + \frac{1}{2} \log n -  \frac{1}{2}\log(2\pi),
\end{multline*}
with $\psi (x) :=  (\log\Gamma(x))^\prime$ the so-called digamma function.  Moreover using that,
 $\psi(x) = \log x - \frac{1}{2x} + o\left(\frac{1}{x}\right)$ as $x\to\infty$, see \cite[6.3.18]{abramowitz}, we obtain that  $\log n - \psi(\frac \beta 2 n+1)$ remains bounded. We get
\eq \label{EHlimit}
 \sup_{n} \left| \frac{1}{n} \E_{Q_{n, \beta}}\left[\H_{\Lambda_n}^{n-\per}(\g)\right] +  \frac{1}{2} \log\frac{n}{2\pi} \right| < + \infty.
\qe
Putting \eqref{Znlimit} and \eqref{EHlimit} together, we obtain \eqref{enttight}, which proves the claim. 

Finally, since the convergence in the local topology is stronger than the weak convergence of point processes, see  \cite[Section~2.1]{GZ93}, any accumulation point of $(Q_{n,\beta})_{n \geq 1}$ has to be $\sineb$. Moreover, \eqref{enttight} also provides that $\sineb$ is a point process, namely charges only simple configurations. Indeed, this yields that $\sineb$ has finite specific entropy and thus has local densities with respect to the Poisson process, see \cite[Chapter15]{Georgiibook}. The proof of the proposition is therefore complete.

\end{proof}

\subsection{From the finite, periodic DLR equations to Theorem \ref{theo:DLRweak} \ref{Thm2Cbis}}
To prove Theorem \ref{theo:DLRweak} \ref{Thm2Cbis}, namely that 
$$
\E_{\sineb}(f-f_{\La, \R}) = 0,
$$ 
we start from Proposition \ref{nperiodicDLR}
$$
\E_{\Qnbeta}(f -  f^{n-\per}_{\La, \La_n}) = 0,
$$
use the convergence of $\Qnbeta$ to $\sineb$ as expressed by Proposition \ref{prop:sineb}, and perform several approximations.

\begin{itemize}
\item [$\diamond$] First, we show that one can replace $ f^{n-\per}_{\La, \La_n}$ by $ f^{n-\per}_{\La, \La_p}$ in the DLR equations for $Q_{n,\beta}$. The contribution of the exterior configuration is indeed negligible when forgetting about the configuration in $\La_n\setminus\La_p$.

\item [$\diamond$] Next, we prove that one can further replace $f^{n-\per}_{\La, \La_p}$ by $f_{\La, \La_p}$, which means we can replace the periodic logarithmic interaction by the usual one up to negligible terms. 

\item [$\diamond$] Finally, we let $n\to\infty$ and replace $Q_{n,\beta}$ by $\sineb$, using  Proposition \ref{prop:sineb}. Moreover, we replace $f_{\La, \La_p}$ by $f_{\La, \R}$ in the remaining DLR equations.
\end{itemize}

This shall complete the proof of Theorem \ref{theo:DLRweak} up to the proof of technical  estimates on the discrepancy, which will be deferred to Section \ref{sec:disc}; preliminary material for proving these estimates is provided in Section \ref{sec:Energy}.

\subsubsection{Step 1: Truncation errors in the periodic DLR setting}
 
The DLR equations \eqref{DLRfinitemodel} obtained for $Q_{n,\beta}$ involve the Gibbs kernel $f_{\La,\La_n}^{n-\per}$, defined in \eqref{def:Llalapn}, where the index $n$ appears twice: as the period of the interaction, and as the size of the window. The following estimate allows us to decouple size and period. 

\begin{proposition}[Truncation error, periodic case]
\label{pperiodicDLR}
 Let $\varepsilon >0$.  For any $p$ large enough (depending on $\La, \epsilon$), for any $n \geq p$, and for any bounded measurable test function $f$ on $\Conf(\R)$, we have
 $$
 \left| \E_{ Q_{n,\beta}}(f^{n-\per}_{\La, \La_n} -  f^{n-\per}_{\La, \La_p})\right| \le \varepsilon \| f\|_\infty\, .
 $$
\end{proposition}

The only difference between $f^{n-\per}_{\La, \La_n}$ and $f^{n-\per}_{\La, \La_p}$ lies in the size of the exterior configuration that is taken into account. To prove Proposition \ref{pperiodicDLR}, we thus need to control the difference of periodic move functions over different large windows. This motivates the following definition.
\begin{definition}[Configurations with small truncation error]
\label{def:smalltruncationfinite}
For any $\delta >0,$ we denote by $\A^{n-\per}_{\La, \La_p}(\delta)$ the set
\begin{equation}
\label{def:AnLaLap}
\A^{n-\per}_{\La, \La_p}(\delta) := \left\{ \g\in\Conf(\R) :\; \sup_{
\begin{subarray}{c}
\eta \in \Conf(\La)\\ 
|\eta|=| \gamma_\La|
\end{subarray} }
 \left| \M^{n-\per}_{\La, \La_n}(\eta, \g) - \M^{n-\per}_{\La, \La_p}(\eta, \gamma) \right| \leq \delta\right\}.
\end{equation}
 \end{definition}
In plain words, if $\gamma$ belongs to $\A^{n-\per}_{\La, \La_p}(\delta)$, then we can change $\gamma_\Lambda$ into any other configuration $\eta$ in $\Lambda$ \textit{with the same number of points}, and the energy cost of this operation as felt by the points in $\Lambda_n \backslash \Lambda_p$ is always less than $\delta$. 
The main ingredient for the proof of Proposition \ref{pperiodicDLR} is that $\A^{n-\per}_{\La, \La_p}(\delta)$ has large  probability under $Q_{n, \beta}$.
\begin{lemma}[The truncation error is often small, finite case]
\label{le:move} For any $\varepsilon, \delta >0,$ we have
 $$ 
 Q_{n, \beta}\left(\,{\A^{n-\per}_{\La,\La_p}(\delta)}\,\right)\geq 1- \varepsilon
 $$
 provided that $p\leq n$ are large enough, depending on $\varepsilon, \delta, \La$.
\end{lemma}
Roughly speaking “the far exterior does not count”. This would be obvious for a short-range interaction, but in the case of the logarithm we need to show that some effective cancellations occur. The proof of Lemma \ref{le:move}  is deferred to Section \ref{subsection:lemove}, and we now prove Proposition \ref{pperiodicDLR}, using Lemma \ref{le:move}

\begin{proof}[Proof of Proposition \ref{pperiodicDLR}]
For any $\g\in\A^{n-\per}_{\La, \La_p}(\delta)$ and $\eta\in\Conf(\La)$ satisfying $|\eta|=|\gamma_\La|$,
  $$ 
  |  Z_{\La, \La_p}^{n-\per}( \g) - Z_{\La, \La_n}^{n-\per}( \g) | \le (\e^{\beta \delta } -1)  Z_{\La, \La_p}^{n-\per}( \g).
  $$ 
  Then, writing $\La^c:=\La_n\setminus\La$ for convenience, we have for any $\g\in\A^{n-\per}_{\La, \La_p}(\delta)$,
\begin{align*}
& \left|f^{n-\per}_{\La, \La_n}(\g) -  f^{n-\per}_{\La, \La_p}(\g) \right| \\
&  = \left | \int f(\eta \cup \g_{\La^c}) \, \L_{\La, \La_n}^{n-\per}(\d \eta, \g)-\int f(\eta\cup\g_{\La^c}) \, \L_{\La, \La_p}^{n-\per}(\d \eta, \g)  \right|  \\
& \le   \left | \int f(\eta \cup \g_{\La^c}) \left(\e^{- \beta (\M_{\La, \La_n}^{n-\per}(\eta, \gamma)-\M_{\La, \La_p}^{n-\per}(\eta, \gamma))}-1 \right) \, \L_{\La, \La_p}^{n-\per}(\d \eta, \g) \right|\\
& \qquad+ \left| \int  f(\eta\cup\g_{\La^c}) \left( \frac{Z^{n-\per}_{\La,\La_p}(\g)-Z^{n-\per}_{\La, \La_n}(\g)}{Z^{n-\per}_{\La, \La_n}(\g)} \right) \, \L_{\La, \La_n}^{n-\per}(\d \eta, \g)\right| \\
&
 \le  (\e^{\beta \delta } -1) \int |f(\eta \cup \g_{\La^c})| \L^{n-\per}_{\La, \La_p}(\d\eta, \g) + (\e^{2\beta \delta } -1)  \int |f(\eta \cup \g_{\La^c})| \L^{n-\per}_{\La,\La_n}(\d\eta, \g) 
 \\& \le  2(\e^{\beta \delta } -1)\|f\|_\infty.
\end{align*}
Given  $\varepsilon >0$, assume that $\delta$ satisfies $2(\e^{\beta \delta}-1)\leq \varepsilon$ and assume further that $n\geq p$ are large enough so that $ Q_{n, \beta}({\A^{n-\per}_{\La,\La_p}(\delta)})\geq 1- \varepsilon$, which is possible thanks to Lemma \ref{le:move}. By using that $\|f_{\La, \La_p}^{n-\per}\|_\infty\leq \|f\|_\infty$ for any $p\leq n$,  we thus obtain
\begin{align*}
 \left| \E_{Q_{n,\beta}}(f_{\La, \La_n}^{n-\per} -  f_{\La, \La_p}^{n-\per})\right|  & \le  \left| \E_{Q_{n,\beta}}\big((f_{\La, \La_n}^{n-\per} -  f_{\La, \La_p}^{n-\per})\bv 1_{\A^{n}_{\La, \La_p}(\delta)}\big)\right| \\
 & \qquad +  2\|f\|_\infty  \Qnbeta\left(\Conf(\R)\setminus \A^{n}_{\La, \La_p}(\delta)\right) \\
 & \le  3\varepsilon \|f\|_\infty
 \end{align*}
 and the lemma follows since $\epsilon$ is arbitrary.
\end{proof}

\subsubsection{Step 2: From periodic to non-periodic interaction}
Next, we show that one can replace the periodic potential $g_n$ by the logarithmic potential $g$ at a small cost.

\begin{proposition}[From periodic to non-periodic potentials]
\label{unper}
Let $\varepsilon >0.$ We have
 $$ 
 \left| \E_{Q_ {n,\beta}}(f_{\La, \La_p} -  f^{n-\per}_{\La, \La_p})\right| \le \varepsilon \| f\|_\infty,
 $$
 provided that $p$ is large enough (depending on $\La, \varepsilon$) and $n\geq p$ is large enough (depending on $p$).
\end{proposition}
This is fairly intuitive: in the Gibbs kernel $f_{\La, \La_p}$, all the interactions take place between points that are at distance at most $p$. The precise value of the period $n$ of the interaction is thus not really important, because for $|x| \leq p \ll n$, we have
$$
\log \left| \sin(\frac{\pi x}{n}) \right| \approx \log \left| \frac{\pi x}{n} \right| \approx \log |x|,
$$
up to an additive constant in the energy, which is irrelevant for a Gibbs specification. The only issue is that this approximation comes with a certain negligible cost \textit{for each pair of points}, so the main ingredient that we will use in the proof of Proposition \ref{unper} is the fact that it is unlikely under $Q_{n, \beta}$ to have too many points in a given bounded set. 
More precisely, if we set 
$$ 
\mathsf{B}_{p} := \left\{\gamma\in\Conf(\R) :\; |\gamma_\La| \le  p, \; |\g_{\La_p}| \le p^2\right\},
$$
the following estimate holds true.

\begin{lemma}[No overcrowding]
\label{le:Bp} For any $\varepsilon >0,$ for $p$ large enough (depending on $\La, \epsilon$), and for $n \geq p$, we have
 $$ 
 Q_{n, \beta}\left(\,\mathsf{B}_{p}\,\right)\geq 1- \varepsilon.
 $$
\end{lemma}
The proof of Lemma \ref{le:Bp}  is deferred to Section \ref{sec:proofleBp}, and we now prove Proposition \ref{unper} using Lemma \ref{le:Bp}.

\begin{proof}[Proof of Proposition \ref{unper}] Let $\delta>0$.
For any $n$ large enough (depending on $\delta$ and $p$) we have, for any $x\neq y \in \Lambda_p,$
$$ 
\left|g_n(x-y) + \log \frac{2\pi}n -g(x-y)\right| = \left|\log \left|\frac{2 \sin \frac{\pi(x-y)}{n}}{\frac{2\pi(x-y)}{n}}\right|\right| \le \frac{\delta}{p^3}.
$$
Let $\g\in\mathsf{B}_p$ and  $\eta \in\Conf(\La)$  satisfying $|\eta|=|\gamma_\La|$ and assume they are simple.  Then, for $n$ large enough (depending on $\delta, p$),  
\begin{align}
\label{aproxper-paper}
\left| \H^{n-\per}_{\La}(\eta) +\frac{1}{2} |\gamma_\La|(|\g_\La|-1) \log \frac{2\pi}{n} - \H_{\La}(\eta) \right| & \le  \frac{\delta}{p^3} \int\d \g_\La^{\otimes 2} \le  \frac{\delta}{p} \leq \delta\\
\label{aproxper-paper2}
\left| \M^{n-\per}_{\La, \La_p}(\eta, \g) - \M_{\La, \La_p}(\eta, \gamma) \right| & \leq  \frac{\delta}{p^3} 2 |\g_\La| \int\d \g_{\La_p\setminus \La} \le 2\delta,
\end{align}
where we used for the second string of inequalities that $|\eta|=|\gamma_\La|.$ It follows,
$$
\big|Z^{n-\per}_{\La, \La_p}( \g)\, \e^{-\frac\beta 2  |\gamma_\La|(|\g_\La|-1) \log \frac{2\pi}{n}} -  Z_{\La, \La_p}( \g)\big| \le (\e^{3\beta\delta}-1)Z_{\La,\La_p}(\g).$$
Moreover, by writing 
\begin{multline*}
  Z_{\La,\La_p}^{n-\per}(\g)\, \e^{-\frac\beta 2 |\gamma_\La|(|\g_\La|-1) \log \frac{2\pi}{n}}f^{n-\per}_{\La, \La_p}(\g)\\=\int f(\eta\cup\gamma_\La)\e^{-\beta(\H_\La^{n-\per}(\eta)+ \frac{1}{2} |\gamma_\La|(|\g_\La|-1) \log \frac{2\pi}{n}+\M^{n-\per}_{\La,\La_p}(\eta,\gamma))} \bv B_{|\g_\La|, \La}(\d \eta)
\end{multline*}
and using again \eqref{aproxper-paper}--\eqref{aproxper-paper2} together with Lemma \ref{le:Bp}, the proposition is obtained by following the same lines than in the proof of Proposition \ref{pperiodicDLR}.
\end{proof}

\subsubsection{Step 3: Truncation errors in the infinite DLR setting}
The results of this section are valid not only for $\sineb$, but for any stationary point process $P$ on $\R$ with finite expected renormalized energy $\E_P[\W(\gamma)]$. We refer the reader to Section \ref{sec:Energy} for a precise definition but for now it is enough to keep in mind that $\E_{\sineb}[\W(\gamma)]<\infty.$ 

\begin{lemma}[Definiteness of the move functions, infinite case]
\label{lem:definimoveinf} Let $P$ be a stationary point process on $\R$ satisfying $\E_P[\W(\gamma)]<\infty$. Then, for $P$-a.e. $\g\in\Conf(\R)$ and every $\eta\in \Conf(\La)$ satisfying $|\eta|=|\gamma_\La|$, the limit
$$
\M_{\La,\R}(\eta,\gamma)=\lim_{p\to\infty}\M_{\La,\La_p}(\eta,\gamma)
$$
exists and is finite, and the convergence is uniform for such $\eta$'s. 
\end{lemma}
Lemma \ref{lem:definimoveinf} is proven in Section \ref{sec:proofdefinimoveinf}

We now state a result concerning the truncation error in the infinite, non-periodic setting. 
\begin{proposition}[Truncation error, infinite setting]
\label{truncinfinite}
Let $P$ be a stationary point process on $\R$ satisfying $\E_P[\W(\gamma)]<\infty$ and let $\varepsilon >0.$ For any $p$  large enough (depending on $\La, \varepsilon$, and $P$), for any bounded measurable function $f$, we have
 $$ 
 \left| \E_{P}(f_{\La, \R} -  f_{\La, \La_p})\right| \le \varepsilon \| f\|_\infty\, .
 $$
\end{proposition}
Proposition \ref{truncinfinite} should be compared to Proposition \ref{pperiodicDLR}. As for the proof of the latter, we rely on a result saying that the truncation error is often small.

The following definition is the counterpart of Definition \ref{def:smalltruncationfinite} in the infinite setting.
\begin{definition}[Infinite configurations with small truncation error]
\label{def:smalltruncationinfinite}
For any $\delta >0$,  set
$$ 
\A_{\La, \La_p}(\delta) := \left\{ \g\in\Conf(\R):\;\sup_{
\begin{subarray}{c}
\eta \in\Conf(\La)\\ 
|\eta|=| \gamma_\La|
\end{subarray} }
 \left| \M_{\La, \R}(\eta, \g) - \M_{\La, \La_p}(\eta, \gamma) \right| \leq \delta\right\}.$$
\end{definition}

\begin{lemma}[The truncation error is often small, infinite case]
\label{lem:truncerrorinfinitecase}
 For any $\varepsilon, \delta >0,$ we have for every $p$ large enough (depending on $P, \varepsilon, \delta, \La$),
 \eq
 \label{PAD}
 P\left(\,{\A_{\La,\La_p}(\delta)}\,\right)\geq 1- \varepsilon.
 \qe
\end{lemma}
Lemma \ref{lem:truncerrorinfinitecase} is the counterpart of Lemma \ref{le:move} in the infinite, non-periodic setting. Its proof is postponed to Section \ref{subsection:lemoveinf}.

\begin{proof}[Proof of Proposition \ref{truncinfinite}] Using Lemma \ref{lem:truncerrorinfinitecase}, the proposition is obtained exactly as in the proof of Proposition \ref{pperiodicDLR}.
\end{proof}

\subsubsection{Proof of the canonical DLR equations}
We may now give the proof of the canonical DLR equations for $\sineb$.
\begin{proof}[Proof of Theorem \ref{theo:DLRweak} \ref{Thm2Cbis}] 

Let $\varepsilon >0$ and $f$ be a bounded Borel local function on $\Conf(\R).$ We write,

\begin{align}
\tag{A}
 \left|\E_{\sineb}\left(f-f_{\La, \R}\right)\right|  & \leq \; \left| \E_{\sineb} (f) -\E_{Q_{n,\beta}}(f)\right|\\
 \tag{B}
  & \quad + \left| \E_{Q_{n,\beta}}(f - f^{n-\per}_{\La, \La_n})\right| \\
  \tag{C}
    &\quad + \left| \E_{Q_{n,\beta}} (f_{\La, \La_n}^{n-\per}) - \E_{Q_{n,\beta}} (f_{\La, \La_p}^{n-\per})\right|\\
      \tag{D}
          &\quad + \left| \E_{Q_{n,\beta}} (f_{\La, \La_p}^{n-\per}) - \E_{Q_{n,\beta}}(f_{\La, \La_p})\right|\\
          \tag{E}
  &\quad + \left| \E_{Q_{n,\beta}} (f_{\La, \La_p}) - \E_{\sineb} (f_{\La, \La_p})\right|\\
  \tag{F}
 & \quad+ \big|\E_{\sineb}( f_{\La, \La_p}-f_{\La, \R}) \big|. 
\end{align}
Proposition~\ref{nperiodicDLR} states that $\text{(B)}=0$, Proposition~\ref{pperiodicDLR} that $\text{(C)}\leq \epsilon\|f\|_\infty$, Proposition~\ref{unper} that $\text{(D)}\leq \epsilon\|f\|_\infty$, and Proposition~\ref{truncinfinite} that  $\text{(F)}\leq \epsilon\|f\|_\infty$, provided that $p$  is chosen large enough (depending on $\La, \epsilon$) and $n$ is large enough (depending on $p$). Moreover, since $f$ and $f_{\La, \La_p}$ are both bounded Borel local functions, it follows from Proposition~\ref{prop:sineb}, that $\text{(A)}$ and $\text{(E)}$ can be made arbitrary small provided that $n$ is large enough (depending on $f, \La, p$), and the proof is complete under the extra assumption that $f$ is a local function. 

In order to extend the result to arbitrary, possibly non local, bounded Borel functions, we proceed as follows. Let $\mathcal{M}$ be the class of all measurable events $A$ such that $\mathbf{1}_{A}$ satisfies the DLR equations, and let $\Pi$ be the class of all measurable events $A$ which are local in the sense that $\mathbf{1}_A$ is a local function as above. So far, we have proven that $\Pi \subset \mathcal{M}$. We want to prove that $\mathcal{M}$ is the whole Borel $\sigma$-algebra. The set $\Pi$ is clearly stable under finite intersections. Moreover, we can check that $\mathcal{M}$ is a monotone class, using monotone convergence and the linearity of DLR equations. By the monotone class theorem, $\mathcal{M}$ contains $\sigma(\Pi)$, the $\sigma$-algebra generated by $\Pi$.  Next, consider the countable collection of open sets $\{\g\in\Conf(\R) : \;|\g\cap(a,b)|<c\}_{a,b,c\in\mathbb Q}\subset\Pi$, which generate the topology of $\Conf(\R)$. Since this collection is countable the $\sigma$-algebra it generates, which is included in $\sigma(\Pi)$, is the whole Borel $\sigma$-algebra. This finally shows that $\mathcal M$ is the whole Borel $\sigma$-algebra of $\Conf(\R)$, and we have thus obtained the DLR equations for any indicator function of a Borel subset of $\Conf(\R)$, which by linearity of the DLR equations extends to every simple function, and finally, by density, to every bounded measurable function.
\end{proof}

\subsection{Renormalized energy and discrepancy estimates}
\label{sec:Energy}

\subsubsection{Renormalized energy}
We gather here the definition of the \textit{renormalized energy},  which is a way to define the logarithmic energy of an infinite point configuration, and some useful properties. A first version of this object was introduced by \cite{sandier2012ginzburg} but we use here the variation introduced in \cite{petrache2017next}.  In the present work, we do not work directly with the energy, we mostly make use of the connection between  the renormalized energy and discrepancy estimates, as explained in the next paragraph.

The following definitions can be found, with more details and justification for existence,  e.g. in \cite[Section 2.6]{Leble:2017mz}.

\begin{definition}[Compatible electric fields]
 For any  $\g\in\Conf(\R)$, a vector field $E:\R^2\to\R^2$ is said to be \textit{an electric field compatible with $\g$}, and we write $E\in\Comp(\gamma)$, if it satisfies:
$$ - \div E = 2\pi (\g - \delta_\R),$$
in the sense of distributions, where by definition, the action of the measure $\delta_{\R}$ on a smooth and compactly supported function $\phi:\R^2\to\R$  is $\int_\R \varphi(\cdot,0) \d x$.
\end{definition}

\begin{definition}[Renormalized energy  of an infinite point configuration] Given any configuration $\gamma\in\Conf(\R)$ and $E\in\Comp(\gamma)$, 
we first consider for any $\eta \in (0,1)$ the \emph{regularized field $E_\eta$}:   we set for any $z=(x,y) \in\R^2,$  
$$ E_\eta(z) := E(z) + \sum_{p \in \g} \nabla f_\eta(x-p,y), \quad \text{ where } f_\eta(z):=\bs 1_{|z|\geq \eta} \log\left|\frac z\eta\right| \text{ for } z\in\R^2.$$
Then, the \emph{renormalized energy of $\gamma$} is defined by
$$  \W(\gamma) := \inf_{E\in\Comp(\gamma)}\left\{\lim_{\eta\to 0}\left(\limsup_{R \rightarrow \infty} \frac 1 R \int_{[-R,R]\times \R} |E_\eta(x,y)|^2 \d x\d y \right)+ 2\pi \log\eta \right\}.$$ 
\end{definition}

For a periodic configuration, the renormalized energy can be computed explicitly in terms of the periodic logarithmic energy of the configuration in a fundamental domain.
\begin{proposition}[Energy of a periodic point configuration] 
\label{wperiodic}
 Given a configuration $\g$ of $n$ distinct points $\g:=\{\g_1, \ldots, \g_n\}\in\Conf(\La_n)$, let $\g^{n-\per}\in\Conf(\R)$ be the $n$-periodic configuration defined by 
 $$
 \g^{n-\per}:=\bigcup_{k_1,\ldots,k_n\in\Z}\big\{\g_1+k_1n,\ldots,\g_n+k_nn\big\}.
 $$ 
 Then, we have
 $$ 
\W(\g^{n-\per}) =   \frac{\pi}n\left( 2\H_{\Lambda_n}^{n-\per}(\g)+n\log\frac{n}{2\pi}\right).
 $$
\end{proposition}
\begin{proof} We refer to  \cite[Proposition 1.5, $d=1$]{petrache2017next} or \cite[Proposition 2.10]{BS}
\end{proof}

\subsubsection{Discrepancy estimates}
 We introduce an important quantity for our purpose: the discrepancy, as well as bounds on the average discrepancy for log-gases.
\begin{definition}[Discrepancy]
\label{def:discrepancy}
The \emph{discrepancy} of $\gamma\in\Conf(\R)$ relative to a bounded Borel set $\Lambda \subset \R$ of Lebesgue measure $|\Lambda|$ is defined by
$$
\Discr_{\Lambda}( \gamma) := |\gamma_\Lambda| - |\Lambda|.
$$\end{definition}

A crucial fact for our purpose is that a bound on the renormalized energy  translates into a discrepancy estimate, see e.g. \cite[Section 3.2]{Leble:2017mz}.
\begin{lemma}[Energy bound yields discrepancy estimate]
\label{lem:energytodiscr}  There exists $C>0$ such that, for any stationary point process $P$ on $\R$ satisfying $\E_P [ \mathbb{W}(\gamma) ]<\infty$ and any bounded Borel set $\La\subset\R$, we have
\begin{equation}
\label{energytodiscr}
\E_P\left[ \Discr_{\Lambda}^2(\g) \right] \leq C\left (C + \E_P \left[ \mathbb{W}(\gamma) \right]\right) |\Lambda|.
\end{equation}
\end{lemma}
\begin{proof}
This is \cite[Lemma 3.2]{Leble:2017mz}.
\end{proof}

In particular, for the periodic log-gas $Q_{n,\beta}$, we obtain the following bound.
\begin{lemma}
\label{le:disc}
 There exists a constant $C_\beta>0$ depending  on $\beta $ only such that, for any bounded Borel set $\La\subset\R,$ and $n\geq 1$ large enough so that $\La\subset\La_n$, we have 
 $$  \E_{Q_{n, \beta}} \left[ \Discr_\La^2(\g)\right] \le C_\beta |\La| \quad 
\textrm{ and } \quad
   \E_{Q_{n, \beta}} \left[ |\g_\La|^2\right] \le 2 |\La|(C_\beta+|\La|).$$
\end{lemma}

\begin{proof}
 As $Q_{n, \beta}$ is the law of a stationary point process, one can apply Lemma \ref{lem:energytodiscr} together with Proposition \ref{wperiodic} to obtain, provided that $\La\subset\La_n$,
 \begin{align*}  
   \E_{Q_{n, \beta}} \left[ \Discr_\La^2(\g)\right] & =  \E_{Q_{n, \beta}} \left[ \Discr_\La^2(\g^{n-\per})\right]\\
   &  \le C\left(C+   \frac{\pi}n \E_{Q_{n, \beta}}\left[2\H_{\Lambda_n}^{n-\per}(\g)+n\log\frac{n}{2\pi}\right] \right) |\La|.
   \end{align*}
Using \eqref{EHlimit}, the first inequality follows.  The second inequality is obtained from the first one by writing
$$
   \E_{Q_{n, \beta}} \left[ |\g_\La|^2\right]\leq 2   \left(\E_{Q_{n, \beta}} \left[ \Discr_\La^2(\g)\right] +|\La|^2\right).
$$
\end{proof}

We will also use the following asymptotic behavior for the discrepancy.

\begin{lemma}
\label{lem:Discrvers0}
If $\gamma\in\Conf(\R)$ satisfies $\mathbb{W}(\gamma)<\infty$ then,  as $k\to\infty$,
$$
\Discr_{[0,k]}(\gamma)=o(k).
$$ 
\end{lemma}
\begin{proof}
This is a consequence of \cite[Lemma 2.1]{petrache2017next}.
\end{proof}

\subsubsection{Proof of Lemma \ref{le:Bp}}
\label{sec:proofleBp}

An easy consequence of Lemma \ref{le:disc} is the following. 

\begin{proof}[Proof of Lemma \ref{le:Bp}] Lemma \ref{le:disc} and Markov inequality yield the existence of $C_\La>0$ such that, for any $1\leq p\leq n$ and  $n$ large enough so that $\La\subset\La_n$, 
\begin{multline*}
Q_{n,\beta}(\Conf(\R)\setminus \mathsf{B}_p)  \leq Q_{n,\beta}(|\gamma_\La|>p)+ Q_{n,\beta}(|\gamma_{\La_p}|>p^2)\\
 \leq \frac{1 }{p^2}\E_{Q_{n,\beta}}(|\gamma_\La|^2)+\frac{1 }{p^4}\E_{Q_{n,\beta}}(|\gamma_{\La_p}|^2) \leq \frac{C_{\La}}{p^2},
\end{multline*}
and the lemma follows. 
\end{proof}

\subsection{Auxiliary proofs}
\label{sec:disc}
We now provide proofs for Lemmas \ref{le:move}, \ref{lem:definimoveinf}, and \ref{lem:truncerrorinfinitecase}.  

In this section, we always assume $p \geq n$ are large enough so that $3\La\subset\La_p\subset\La_n$.  We  use the following notation:
\begin{itemize}
\item[$\diamond$] The distance of $x\in\R$ to a subset $\La\subset\R$ is $\dist(x,\La):=\inf_{y\in\La}|x-y|$, and the distance from a subset $I\subset\R$ to $\La$ is $\dist(I,\La):=\inf_{x\in I}\dist(x,\La)$.
\item[$\diamond$] For any $\gamma_1,\gamma_2\in\Conf(\R)$ with $|\gamma_1|=|\gamma_2|=M<\infty$, say $\gamma_j=\sum_{i=1}^M \delta_{\gamma_j^i}$, we set 
\begin{equation}
\label{def:W1}
W_1(\gamma_1,\gamma_2):=\inf_{\sigma\in\frak S_M}\sum_{i=1}^{M}|\gamma_1^i-\gamma_2^{\sigma(i)}|
\end{equation}
where $\frak S_M$ is the set of permutations of $\{1, \ldots, M\}.$
Note that the definition does not depend on the indexing.
\item[$\diamond$]We denote by $\Leb$ the Lebesgue measure of $\R$  and by $\Leb_\La$ its restriction to $\La\subset\R$. 
\end{itemize}

\subsubsection{Intermediary results}
\begin{lemma}[The electrostatic potential generated when moving points]
\label{le:estimove}
 Take any configurations $\gamma,\eta$ in $\Conf(\R)$  such that $|\eta|=|\gamma_\La|$.  Recalling the definitions \eqref{def:g} and \eqref{def:gng}, we set  for convenience
\eq
\label{Psidef}
\Psi_n:=g_n*(\eta-\gamma_\La), \quad \Psi:=g*(\eta-\gamma_\La).
\qe
Given $\epsilon>0$, the following holds true for $p$ large enough (depending on $\epsilon$) and $n$ large enough (depending on $p$).
\begin{itemize}
 \item[(a)]  
 \begin{equation*}
 \left|\int_{\La_p} \Psi(s) \d s \right| \le \epsilon \, W_1(\eta, \gamma_\La) 
 \end{equation*}
\item[(b)]

\begin{equation*}
\left| \int_{\La_n \setminus \La_p}  \Psi_n(s) \d s \right| \le  \varepsilon \Big( W_1(\eta,\gamma_\La)+ |\gamma_\La|\Big).
\end{equation*}
 \item[(c)] For any $x \in \La_n \setminus \La_p,$
 \begin{equation*}
|\Psi_n(x)|\leq \frac{W_1(\eta,\gamma_\La) }{\dist(x,\La)} \qquad \mbox{ and }\qquad |\Psi_n^\prime(x)|\leq \frac{8 W_1(\eta,\gamma_\La)}{\dist(x,\La)^2}.
\end{equation*}
\item[(d)]   For any $x \in \R \setminus \La_p,$
 \begin{equation*}
|\Psi(x)|\leq \frac{W_1(\eta,\gamma_\La) }{\dist(x,\La)} 
\qquad \mbox{ and }\qquad
|\Psi^\prime(x)|\leq \frac{8 W_1(\eta,\gamma_\La)}{\dist(x,\La)^2}.
\end{equation*}
\end{itemize}
\end{lemma}

\begin{proof}
Let us enumerate the configurations as $\gamma_\La = \sum_{i=1}^M \delta_{\gamma_i}$ and $\eta = \sum_{i=1}^M \delta_{\eta_i}$.  To prove (a), we start by writing 

$$ 
\int_{\La_p} \Psi(s) \d s  = \sum_{i=1}^M \int_{\La_p} \Big(\log |\gamma_i -s| - \log |\eta_i -s|  \Big) \d s.
$$
Now, set 
$$
\V(t) := \int_{-1}^1 - \log |t - s| \d s= (1+t) \log (1+t) + (1-t) \log (1-t).
$$
and let $k>0$ be fixed so that $\La\subset\La_k$. We obtain, by a linear change of variables sending $[-1,1]$ on $\La_p = [-\tfrac p2, \tfrac p2]$:
$$
\left|\int_{\La_p} \Psi(s) \d s \right|= \left|\frac{p}{2} \sum_{i=1}^M \left( \V\left(\frac{2\eta_i}{p}\right)-  \V\left(\frac{2\gamma_i}{p}\right)\right)\right|\leq \sup_{[-\frac kp,\frac kp]} |\V'|\sum_{i=1}^M|\eta_i-\gamma_i|.
$$
Since $\V^\prime$ is continuous near the origin, since $\V'(0)=0$ and since the enumeration of $\eta$ and $\gamma_\La$ is arbitrary, (a) follows by taking $p$ large enough, depending on $\La$ and $\epsilon$.

We now turn to (b).  Since $g_n*{\rm Leb}_{\La_n}=0$  on $\La_n$, see e.g. \cite[Equation (2.49)]{BS}, we have 
$$
 \int  \Psi_n\, {\rm Leb}_{\La_n \setminus \La_p} = - \int   \Psi_n \,{\rm Leb}_{ \La_p}.
 $$
 For any fixed $p$, we have for $n$ large enough and for any $x\neq y \in \La_p,$
 \begin{equation}
 \label{loggn}
 \left|g_n(x-y) + \log \left( \frac{2\pi}{n} \right) -g(x-y)\right| = \left|\log \left|\frac{2 \sin \frac{\pi(x-y)}{n}}{\frac{2\pi(x-y)}{n}}\right|\right| \le \frac{1}{p^2}.
 \end{equation}
We may thus write 
$$
  \left|\int  \Psi_n\, {\rm Leb}_{\La_n \setminus \La_p}\right| = \left|\int \Psi_n\, {\rm Leb}_{\La_p}\right|  \le  \left|\int \Psi\, {\rm Leb}_{\La_p}\right|  + \frac{|\gamma_\La| p}{p^2},
  $$
 where we have used the fact that $\eta$ and $\gamma_\La$ have the same number of points, hence the contribution of the constant term $\log \left( \frac{2\pi}{n} \right)$ in \eqref{loggn} vanishes. Using point (a) of the present lemma, we obtain
 $$
  \left|\int  \Psi_n\, {\rm Leb}_{\La_n \setminus \La_p}\right|
 \le  \varepsilon \left( W_1(\eta,\gamma_\La)+ |\gamma_\La|\right),
$$
for $p$ large enough (depending on $\La, \epsilon$) and $n$ large enough (depending on $p$).

Finally, we prove (c) and (d). For any $x\in\Lambda_n \backslash \Lambda_p$, by applying the mean value theorem to $g_n(x-\cdot)$ between $\eta_i$ and $\gamma_i$, we obtain
$$
\left|g_n(x-\eta_i) - g_n(x-\gamma_i)\right| \leq |\eta_i-\gamma_i| \frac{\pi}{n} \frac{1}{|\tan\left(\frac{\pi}{n} \dist(x, \Lambda)\right)|} \leq \frac{|\eta_i-\gamma_i|}{\dist(x, \Lambda)},
$$
and the first inequality of (c) is obtained by summing over $i \in \{1, \dots, M\}$. We obtain the second inequality of (c), as well as (d),  by the same argument but using $g_n',g$ and $g'$ instead of $g_n$ respectively.

\end{proof}

\subsubsection{Proof of Lemma \ref{le:move}}
\label{subsection:lemove}
We want to show that in the finite periodic model, with high probability, it is possible to move the points in $\La$ at a small cost.

\begin{proof}[Proof of Lemma \ref{le:move}] In view of \eqref{Psidef}, and in order to obtain the controls on the move functions from discrepancy estimates, it is convenient to work with ``move functions with background'' defined by

\begin{equation}
\label{def:Movetilde}
\Move^{n-\per}_{\La, \La_p}(\gamma,\eta)  :=\int _{\La_p\setminus\La}\Psi_n\, \d (\gamma- \Leb).
\end{equation}
They are related to the usual move functions defined in \eqref{def:Mnlambda} as follows:
\begin{equation}
\label{MnpMnptilde}
\M^{n-\per}_{\La, \La_p}(\gamma,\eta) = \Move^{n-\per}_{\La, \La_p}(\gamma,\eta) + \int_{\La_p\setminus\La} \Psi_n  \,\d\Leb.
\end{equation}
Given any $\gamma\in\Conf(\R)$, let us set for convenience,
\begin{align*}
\mathsf E & := \sup_{
\begin{subarray}{c}
\eta \in \Conf(\La)\\ 
|\eta|=| \gamma_\La|
\end{subarray} }
\left|\M^{n-\per}_{\La, \La_n}(\gamma,\eta)-\M^{n-\per}_{\La, \La_p}(\gamma,\eta)\right|,
\\ 
 \widetilde{\mathsf E} & := \sup_{
\begin{subarray}{c}
\eta \in \Conf(\La)\\ 
|\eta|=| \gamma_\La|
\end{subarray} }\left|\Move^{n-\per}_{\La, \La_n}(\gamma,\eta)-\Move^{n-\per}_{\La, \La_p}(\gamma,\eta)\right| .
\end{align*}
Using \eqref{MnpMnptilde}  we see that
$$
\mathsf{E} \leq \widetilde{\mathsf E} + \sup_{
\begin{subarray}{c}
\eta \in \Conf(\La)\\ 
|\eta|=| \gamma_\La|
\end{subarray} } \left|\int_{\La_n \setminus \La_p} \Psi_n \, \d \Leb\right|.
$$
Using Lemma \ref{le:estimove}(b) and the fact that $W_1(\eta,\gamma_\La)\leq |\La|\,|\g_\La|$, given any $\alpha>0$, we have, if $p$ is large enough (depending on $\La, \alpha$) and $n$ is large enough (depending on $p$)

$$
\mathsf E 
\leq \widetilde{\mathsf E}+\alpha |\gamma_\La|.
$$ 

Next, for any $\delta,L>0$, we have
\begin{align*}
 \Qnbeta\Big(\Conf(\R)\setminus \A_{\La, \La_p}^{n-\per}(\delta)\Big) &  =  \Qnbeta\big( \mathsf E>\delta\big)\\
&   \leq   \Qnbeta({\mathsf E}>\delta ,\; |\g_\La|\leq L)+  \Qnbeta(|\g_\La|> L)\\
&  \leq   \Qnbeta(\widetilde{\mathsf E}>\delta-\alpha L  ,\; |\g_\La|\leq L)+  \Qnbeta(|\g_\La|> L)\\
 &  \leq  \frac1{\delta-\alpha L}\E_{\Qnbeta}\big[\widetilde{\mathsf E}\, \bs 1_{|\gamma_\La|\leq L}\big]+\frac1{L^2}\E_{\Qnbeta}[|\gamma_\La|^2].
\end{align*}
To prove the lemma, it is enough to show that, given any $L>0$,
$$
\E_{\Qnbeta}\big[\widetilde{\mathsf E}\, \bs 1_{|\gamma_\La|\leq L}\big]
$$
can be made arbitrarily small by taking first $p$, then $n$, large enough.  Indeed, given any $\delta>0$, by taking $p,n,L$ large enough and $\alpha:=\delta/2L$  the lemma would follow from Lemma~\ref{le:disc}.

To prove this claim, we split $\La_n\setminus\La_p$ into the subintervals  
$$
I_j:= 
\begin{cases}                                                                                     
\left(\frac{j}{2},  \frac{j+1}{2} \right] & \textrm{ if } j > 0 \\
\left[\frac{j}{2},  \frac{j+1}{2} \right) & \textrm{ if } j < 0,
\end{cases}
$$
so as to write
\eq
\label{Mex1}
\Move^{n-\per}_{\La, \La_n}(\gamma,\eta)-\Move^{n-\per}_{\La, \La_p}(\gamma,\eta) 
=\left(\sum_{j=-n}^{-p-1}+\sum_{j=p}^{n-1}  \right)\int_{I_j} \Psi_n\ \d(\gamma-\leb).
\qe
By applying the mean value theorem to $\Psi_n$ and  Lemma \ref{le:estimove}(c), we obtain for  any $x\in I_j$,
$$
\left| \Psi_n(x) - \Psi_n\left(\frac{j}{2}\right) \right| \leq  \frac{8 W_1(\eta,\gamma_\La)}{\dist(x,\La)^2}, 
$$
and thus,  for any $-n\leq j<n$,
$$
\left|\int_{I_j}\Psi_n\ \d(\gamma-\leb)-\Psi_n\left(\frac j2\right) \Discr_{I_j}(\gamma)\right|\leq   \frac{8 W_1(\eta,\gamma_\La)}{\dist(I_j,\La)^2}\ \big(\Discr_{I_j}(\gamma)+1\big).
$$
Since $p$ is arbitrarily large and $\La$ is fixed, there exists $c>0$ such that, for any $|j|\geq p$,
$$
\dist(I_j,\La)\geq \frac jc.
$$
Combined with \eqref{Mex1}, we obtain
\begin{multline}
\label{bouge}
\left|\Move^{n-\per}_{\La, \La_n}(\gamma,\eta)-\Move^{n-\per}_{\La,\La_p}(\gamma,\eta) -\left(\sum_{j=-n}^{-p-1}+\sum_{j=p}^{n-1}\right)\Psi_n\left(\frac j2\right) \Discr_{I_j}(\g)\right|\\
\leq 8c^2W_1(\eta,\gamma_\La) \left(\sum_{j=-n}^{-p-1}+\sum_{j=p}^{n-1}\right) \frac{\Discr_{I_j}(\g)+1}{j^2} .
\end{multline}
By performing a summation by parts, we can write
\begin{align}
\sum_{j=p}^{n-1}\Psi_n\left(\frac j2\right) \Discr_{I_j}(\g) & =\Psi_n\left(\frac n2\right) \Discr_{(0,\frac{n}{2}]}(\g)-\Psi_n\left(\frac p2\right) \Discr_{(0,\frac{p}{2}]}(\g) \nonumber\\ & \qquad+ \sum_{j=p}^{n-1}\left(\Psi_n\left(\frac {j}2\right)-\Psi_n\left(\frac {j+1}2\right)\right)\Discr_{(0,\frac{j+1}{2}]}(\g).
\end{align}
Using again the mean value theorem and Lemma \ref{le:estimove}(c), we have
\eq
\left|\sum_{j=p}^{n-1}\left(\Psi_n\left(\frac {j}2\right)-\Psi_n\left(\frac {j+1}2\right)\right)\Discr_{(0,\frac{j+1}{2}]}(\g)\right|  
\le 8 c^2W_1(\eta,\gamma_\La)  \sum_{j=p}^{n-1} \frac{\left|\Discr_{(0,\frac{j+1}{2}]}(\g)\right|}{j^2}
\qe
and similar estimates holds for the sum where $j$ ranges from $-n$ to $-p-1$. Moreover, since $x/\dist(x,\La)$ is bounded when $x\notin 2\La$, it follows from Lemma \ref{le:estimove}(c) that there exists $\kappa>0$ independent on $\eta,\gamma$ such that 
\eq
\sup_{x\notin 2\La}|x\Psi_n(x)|\leq \kappa W_1(\eta,\gamma_\La).
\label{bouge2}
\qe
As a consequence, we obtain from  \eqref{bouge}--\eqref{bouge2},
\begin{eqnarray}
\label{bouge2la}
\left|\Move^{n-\per}_{\La, \La_n}(\gamma,\eta)\right.  - \, & \left. \Move^{n-\per}_{\La, \La_p}(\gamma,\eta)\right|  \nonumber
 &\le 2\kappa W_1(\eta,\gamma_\La) \left( \frac{\left| \Discr_{(0,\frac{n}{2}]}(\g)\right|}n+\frac{\left| \Discr_{(0,\frac{p}{2}]}(\g)\right|}p\right)\\ 
&&+ 2\kappa W_1(\eta,\gamma_\La) \left( \frac{\left| \Discr_{(-\frac{n}{2},0]}(\g)\right|}n+\frac{\left| \Discr_{(-\frac{p}{2},0]}(\g)\right|}p\right)\nonumber\\
&+  8 c^2W_1(\eta,\gamma_\La) & \,\, \left(\sum_{j=p}^{n-1} \frac{\left|\Discr_{(0,\frac{j+1}{2}]}(\g)\right|}{j^2} + \sum_{j=-n}^{-p-1} \frac{\left|\Discr_{(\frac{j+1}{2},0]}(\g)\right|}{j^2} \right)\nonumber\\
&& +  8 c^2W_1(\eta,\gamma_\La)\left(\sum_{j=-n}^{-p-1}+\sum_{j=p}^{n-1}\right) \frac{\Discr_{I_j}(\g)+1}{j^2}\nonumber\\
& &=: W_1(\eta,\gamma_\La) \, \Err_{n,p}(\gamma),
\end{eqnarray}
%with an error term 
%\begin{equation}
%\label{defF}
%\Err_{n,p}(\gamma) :=8 c^2 \left(\sum_{j=-n}^{-p-1}+\sum_{j=p}^{n-1}\right) \frac{\Discr_{I_j}(\g)+1}{j^2}.
%\end{equation}
Given any $L>0$, we obtain from  \eqref{bouge2la} and the upper bound $W_1(\eta,\gamma_\La)\leq |\gamma_\La|\,|\La| $ that
\eq
\label{finalboundE}
\E_{\Qnbeta}\big[\widetilde{\mathsf E}\, \bs 1_{|\gamma_\La|\leq L}\big]\leq L  |\La| \,\E_{Q_{n,\beta}}\big[ \Err_{n,p}(\gamma) \big].
\qe

Finally, we use Cauchy-Schwarz inequality and Lemma \ref{le:disc}, to obtain
$$  
\E_{Q_{n,\beta}}\left(\frac{|\Discr_{(0,\frac{n}{2}]}(\g)|}{n}\right) \le  \sqrt{\frac{C_\beta}{2 n}}.
$$
and
$$ \E_{Q_{n,\beta}}\left( \sum_{j=p}^{n-1} \frac{\left|\Discr_{(0,\frac{j+1}{2}]}(\g)\right|}{j^2}\right) \le \sqrt{\frac{C_\beta}2} \sum_{j=p}^\infty \frac{\sqrt {j+1}}{j^2}, $$
and this yields that 
$\E_{Q_{n,\beta}}\big[\Err_{n,p}(\gamma)\big]$ can be made arbitrarily small when $p\leq n$ are large enough. The lemma follows from \eqref{finalboundE}.

\end{proof}

\subsubsection{Proof of Lemma \ref{lem:definimoveinf} and Lemma \ref{lem:truncerrorinfinitecase}}
\label{sec:proofdefinimoveinf}
\label{subsection:lemoveinf}

\begin{proof}
Let $P$ be a stationary point process such that $\E_P[\W(\gamma)]<\infty$. We start by showing that, for $P$-a.e. $\gamma\in\Conf(\R)$ and  every $\eta\in\Conf(\La)$ such that $|\eta|=|\gamma_\La|$, the sequence $\{\M_{\La, \La_p}(\eta, \gamma)\}_{p \geq 1}$ is a Cauchy sequence, thus proving Lemma \ref{lem:definimoveinf}. 

As in the proof of Lemma \ref{le:move}, we introduce the move functions with  background,
$$
\Move_{\La,\La_p}(\gamma,\eta):=\int_{\La_p\setminus\La} \Psi \, \d(\gamma-\Leb).
$$
Since by definition,
$$
\M_{\La, \La_p}(\gamma,\eta) = \Move_{\La, \La_p}(\gamma,\eta) + \int \Psi \,\Leb_{\La_p \setminus \La},
$$
it follows from Lemma \ref{le:estimove}(a) that it is enough to show that $\{\Move_{\La, \La_p}(\eta, \gamma)\}_{p \geq 1}$ is a Cauchy sequence, uniformly in $\eta$.  For any $m\geq p$, by following the same steps as in the proof of Lemma \ref{le:move}, using estimates on $\Psi$ instead of $\Psi_n$, we have with $\Err_{m,p}(\gamma)$ defined in \eqref{bouge2la}, 
\eq
\label{bouge2lainf}
\left|\Move_{\La, \La_m}(\gamma,\eta)-\Move_{\La, \La_p}(\gamma,\eta)\right|  \leq W_1(\eta,\gamma)\,\Err_{m,p}(\gamma) \leq |\gamma_\La| |\La|\,\Err_{m,p}(\gamma).
\qe
Now, using Lemma \ref{lem:energytodiscr}, we obtain
$$ 
\E_P\left( \sum_{j=p}^\infty \frac{\left|\Discr_{(0,\frac{j+1}{2}]}(\g)\right|}{j^2}\right) \le \sqrt{C(C+ \E_P(\mathbb{W}))} \sum_{j=p}^\infty \frac{\sqrt{ j+1}}{j^2} <\infty
$$
so that, 
$$
\sum_{j=p}^\infty \frac{\left|\Discr_{(0,\frac{j+1}{2}]}(\g)\right|}{j^2} < + \infty  \qquad P\text{-a.s}
$$ 
and in particular,
$$ 
\lim_{ p \to \infty} \sum_{j=p}^{\infty} \frac{\left|\Discr_{(0,\frac{j+1}{2}]}(\g)\right|}{j^2} = 0 \qquad P\text{-a.s}.$$
By similar arguments, we have 
$$ 
\lim_{p \to \infty} \sum_{j=p}^{\infty} \frac{|\Discr_{I_j}(\g)|+1}{j^2} = 0 \qquad P\text{-a.s},
$$
and the same holds true for the sums involving negative $j$'s. Moreover, since Lemma~\ref{lem:Discrvers0} yields that 
$$
\lim_{m\to\infty}\frac{\left| \Discr_{(0,\frac{m}{2}]}(\g)\right|}m = 0 \qquad P\text{-a.s},
$$
we have obtained that
$$
\lim_{p\to\infty}\lim_{m\to\infty} \Err_{m,p}(\gamma) =0 \qquad P\text{-a.s}
$$
and our claim follows from \eqref{bouge2lainf}.

Moreover, the previous estimates show that,
$$
\left|\Move_{\La, \R}(\gamma,\eta)-\Move_{\La, \La_p}(\gamma,\eta)\right|  \leq  |\gamma_\La| |\La|\,\Err_{\infty,p}(\gamma)
$$
where
\begin{align}
\Err_{\infty,p}(\gamma)& := 2\kappa  \left(\frac{\left| \Discr_{(0,\frac{p}{2}]}(\g)\right|}p+\frac{\left| \Discr_{(-\frac{p}{2},0]}(\g)\right|}p\right)\\ 
&+  8 c^2 \left(\sum_{j=p}^{\infty} \frac{\left|\Discr_{(0,\frac{j+1}{2}]}(\g)\right|}{j^2} + \sum_{j=-\infty}^{-p-1} \frac{\left|\Discr_{(\frac{j+1}{2},0]}(\g)\right|}{j^2} \right)\nonumber\\
& +  8 c^2\left(\sum_{j=-\infty}^{-p-1}+\sum_{j=p}^{\infty}\right) \frac{\Discr_{I_j}(\g)+1}{j^2}\nonumber\\
\end{align}
exists $P$-a.s. Now the proof of \eqref{PAD} is exactly the same as in Lemma \ref{le:move} but using the estimates on $\Psi$ instead of $\Psi_n$.
\end{proof}

\section{Number-rigidity for solutions of canonical DLR equations}
\label{sec:rigidity}
In this section we prove that $\sineb$ is number-rigid in the sense of \cite{ghosh2017rigidity}, that is  part \ref{Thm1A} of Theorem~\ref{theo:DLRmain}. In fact, we prove that any stationary process satisfying  \eqref{canonicalDLR} is number-rigid, which is the main result of this section. We say that a point process is \emph{stationary} when it is invariant under translations of the configurations $\gamma\mapsto \gamma+x:=\{x+y:\;y\in\gamma\}$ for any $x\in\R$ . Recall also that $f_{\La,\R}$ was introduced in~\eqref{fKernel}.

\begin{definition}[Canonical DLR] Let us fix $\beta>0$. We say that a stationary point process $P$ on $\R$  satisfies the \emph{canonical DLR equations} if $\E_P[\mathbb W(\g)]<\infty$ and  
\begin{equation}
\tag{canonical DLR}
\label{canDLR}
\E_{P}(f-f_{\La, \R}) = 0
\end{equation}
for every bounded Borel set $\Lambda\subset\R$ and every bounded Borel function $f:\Conf(\R)\to\R$.
\end{definition}

The assumption that $P$ has finite renormalized energy $\E_P[\mathbb W(\g)]$ is here to ensure that the move functions and thus $f_{\La,\R}$ are well-defined. Our goal is now to prove:

\begin{theorem} 
\label{Theo-rigidity}
If  $P$ is a stationary point process on $\R$ that satisfies \ref{canDLR}, then $P$ is number-rigid.
\end{theorem}
%In particular, we obtain that for any $\beta > 0$, the $\sineb$ process is number-rigid, as stated in Theorem \ref{theo:rigidity}(a)

Thus, Theorem~\ref{theo:DLRmain}\ref{Thm1A} follows from Theorem~\ref{theo:DLRweak}\ref{Thm2Cbis} and Theorem~\ref{Theo-rigidity}. As a consequence of Theorem~\ref{theo:DLRmain}\ref{Thm1A}, Theorem~\ref{theo:DLRweak}\ref{Thm2Cbis} upgrades to Theorem~\ref{theo:DLRmain}\ref{Thm1C} and, since Theorem~\ref{theo:DLRmain}\ref{Thm1B} has already been proven in the previous section, the proof of our main theorem is complete.

  The proof of Theorem \ref{Theo-rigidity} is based on canonical and grand canonical descriptions of a Gibbs point process via its Campbell measures. These descriptions have been studied intensively in the seventies and  eighties, see for instance \cite{GeorgiiLN79,Kozlov,Nguyen-Zessin,Wakolbinger-Eder}. The proof goes by contradiction and follows three steps:

\paragraph{Step 1:} First, we show in Section~\ref{sec:Campbell} that any point process satisfying the canonical DLR equations admits a canonical description via its Campbell measures, see Theorem~\ref{representationCPn}. 

\paragraph{Step 2:}  Next, if we further assume that the process is ergodic and, for the sake of contradiction, \textbf{not} number-rigid, then we show in Section~\ref{sec:rigidergo} that this representation can be extended into a {grand canonical} version, see Theorem~\ref{Theoabsurde}. 

\paragraph{Step 3:} Finally, we show that the grand canonical representation yields a contradiction because of the long range of the logarithmic interaction, see Section \ref{sectionconclusionproof}. Roughly speaking, we use this representation to move points far away from the origin and show that the configurations obtained should have a much larger weight than it is allowed.\\

It turns out these three steps can be performed in a much more general setting than the one dimensional logarithmic interaction. This leads to a more general result than Theorem~\ref{Theo-rigidity} that we present in Section~\ref{Sec:general}, see Theorem~\ref{DLRrigidityGeneral}; we focused on the one dimensional log-gas for the sake of the presentation. 

\paragraph{Convention:}For convenience, if $\bv x_n=(x_1,\ldots,x_n)\in\R^n$, given any $\gamma\in\Conf(\R)$ with an abuse of notation we write $\gamma\setminus \bv x_n$  (resp. $\gamma\cup \bv x_n$, etc) instead of $\gamma\setminus \{\xn\}$ (resp. $\gamma\cup \{\xn\}$, etc).
%Conversely, given a configuration of $n$ points denoted by $\{\bv x_n\}$, then $\bv x_n$ is by convention the element of $\R^n$ given by \emph{ordering} the points of the configuration $\{\bv x_n\}$.  
Moreover, given $\gamma\in\Conf(\R)$,  the sum
$$
\sum_{\bv x_n\subset\gamma}
$$
means that we are summing over all ordered  $n$-tuples of points from the configuration~$\gamma$, namely
$$
\sum_{\bv x_n \subset\gamma} f(\bv x_n)=\sum_{ \substack{x_1,\ldots,x_n\in\gamma\\x_i\neq x_j \text{ for } i\neq j}} f(x_1,\ldots,x_n).
$$
The reader should have in mind that in the following the configuration $\g$ will be simple since it comes from a point process.  

\subsection{A first consequence of the canonical DLR equations}

The following corollary will be useful in the sequel.

\begin{corollary}
\label{DLRcharge} Let $P$ be a stationary point process on $\R$ that satisfies \ref{canDLR}.  Then, for any disjoint bounded Borel sets $B_1,\ldots,B_k\subset\R$ with positive Lebesgue measure and any integers $n_1,\ldots,n_k$,  we have
$$
P\big(|\g_{B_1}|=n_1,\ldots,|\g_{B_k}|=n_k\big)>0.
$$
\end{corollary}
In particular, Corollary~\ref{sinebcharge} now follows from Theorem~\ref{theo:DLRweak}.

\begin{proof} Set $n:=\sum_jn_j$.  Let $\La\subset\R$ be a bounded Borel set so that $\Lambda \setminus (B_1 \cup \ldots \cup B_k)$ has positive Lebesgue measure
and  
$P(|\g_{\La}|\geq n)>0$. The latter is ensured as soon as $\La$ is large enough since $P$ is stationary and in particular $\E_P|\g|=+\infty$. If we set 
$f(\g):=\prod_{j=1}^k \1_{|\g_{B_j}|=n_j}$, then by applying \ref{canDLR} on $\La$ we have
$$ P\big(|\g_{B_1}| =n_1,\ldots,|\g_{B_k}|=n_k\big)=  \int  \prod_{j=1}^k f(\{x_1,\ldots,x_{|\g_\La|}\})\rho_{\La^c}(x_1,\ldots,x_{|\g_\La|})\prod_{i=1}^{|\g_\La|}\d x_i \,P(\d\g).
$$
Now, if $P\big(|\g_{B_1}| =n_1,\ldots,|\g_{B_k}|=n_k\big)=0$, then  for $P$-a.e. $\g$,
$$
\int \prod_{j=1}^k f(\{x_1,\ldots,x_{|\g_\La|}\})\rho_{\La^c}(x_1,\ldots,x_{|\g_\La|})\prod_{i=1}^{|\g_\La|}\d x_i=0,
$$ 
but by the definition of $\rho_{\La^c}$, see \eqref{def:rhoweak}, this yields that $|\g_\La|<n$ for $P$-a.e $\g$, which is not possible. \end{proof}

\subsection{Campbell measures}
\label{sec:Campbell}
In this section, we introduce the Campbell measures and prove that the canonical DLR equations yield a result on the representation  for these measures.

\begin{definition}[Campbell measure] The \emph{Campbell measure} of order $n\ge 1$ of a point process $P$ on $\R$ is the measure $\CP^{(n)}$ on $\R^n\times \Conf(\R)$  defined by, for any positive Borel test function $f$,
\begin{equation}
\label{def:CPn}
\CP^{(n)}(f):=\int \sum_{
\bv x_n \subset\g
 } f(\xn,\g\setminus\xn) \,P(\d\g).
 \end{equation}
%where the sum ranges over all ordered distinct $n$-tuple of points $x_1,x_2,\ldots,x_n$ of $\g$.
\end{definition}
The Campbell measure encodes the joint distribution of $n$ typical points $\xn$ and their neighborhood $\g \setminus\xn$ for a random configuration $\g$ with law $P$. For example, in the case of a Poisson point process $\Pi$ of intensity measure $\rho$, the Campbell measure is given by 
$$
\mathsf{C}_{\Pi}^{(n)}= \rho^{\otimes n} \otimes \Pi,
$$
which is known as the Slivnyak-Mecke Theorem, see e.g. \cite[Section 3.2]{moller2003statistical}.

In the following, we need to introduce the following \emph{cost function}.

\begin{definition}[Cost of moving $n$ points from $0$ in $\gamma$]
\label{definitionhn} Let $\g\in\Conf(\R)$ and  take $\bv x_n\subset\g$. We consider the cost of moving  the $n$-tuple $(0, \dots, 0)$ to $\bv x_n=(x_1,\ldots,x_n)$ defined by
\begin{equation}
\label{costxg}
 \h(\bv x_n,\g) = \sum_{i < j} g(x_i-x_j) + \lim_{p\to \infty} \int_{-p}^p \sum_{i=1}^n \big(g(x_i-y) - g(y)\big)\g(\d y)
 \end{equation}
provided the limit exists. Note that $ \h(\bv x_n,\g) $ does not depend on the ordering of $\bv x_n$.
\end{definition}

\begin{remark}
If $P$ is a stationary process satisfying $\E_P[\W(\gamma)] < + \infty$, then the limit  \eqref{costxg} exists for $P$-a.e. $\gamma$ and any $\xn$ (the proof is the same as that of Proposition \ref{lem:definimoveinf}).  
\end{remark}

The following theorem is a first description of the structure of Campbell measures for point processes satisfying canonical DLR equations.
\begin{theorem}
\label{representationCPn}
Let $P$ be a stationary point process on $\R$ satisfying  \ref{canDLR}. Then, for every $n\ge 1$, there exists a Borel measure $\Q_n$ on $\Conf(\R)$ such that $\CP^{(n)}$ is absolutely continuous with respect to $\Leb^{\otimes n} \otimes \Q_n$ and has density
\begin{equation}
\label{CPnrep1}
\frac{ \d\CP^{(n)}}{\d\Leb^{\otimes n} \otimes \Q_n} (\xn,\g) = \e^{-\beta \h(\xn,\g)}.
\end{equation}
\end{theorem}

\begin{proof}
Consider the tilted measure,
\begin{equation}
\label{def:tCP}
\tCP^{(n)} := \e^{\beta \h}\, \CP^{(n)}.
\end{equation}
It is thus enough to show that $\tCP^{(n)}=\Leb^{\otimes n} \otimes \Q_n$ for some measure $\Q_n$ on  $\Conf(\R)$.

Let $f:\R^n \times \Conf(\R)\to[0,\infty)$ be a measurable function satisfying $f(\bv x,\gamma)=0$ as soon as $\bv x\notin K$ for some compact set $K\subset\R^d$. Let $\Lambda \subset \R$ be a bounded Borel set such that $K\subset\La^n$. Using the definition \eqref{def:CPn}, \ref{canDLR}, and setting
$$
F_\La(\xn,\eta,\g) := \H_{\La}(\eta)+\M_{\La,\R}(\eta, \g) - \h(\xn,\eta\cup \g_{\La^c} \setminus\xn),
$$
we find
\begin{align*}
\label{CampbellABC}
\tCP^{(n)}(f) & = \int \sum_{\xn\subset\g_\Lambda } 
\e^{\beta \h(\xn,\g\setminus\xn)}f(\xn, \g\setminus \xn) P(\d\g)\\
& = \int   \sum_{\xn\subset\eta } 
\e^{\beta \h(\xn,\,\eta\,\cup\, \g_{\La^c}\setminus\xn)} f(\xn, (\eta \setminus \xn) \cup \gamma_{\La^c})\\
&\qquad \times \frac{1}{Z_{\La,\R}(\g)}\e^{- \beta (\H_{\La}(\eta)+\M_{\La,\R}(\eta, \g))} \,\B_{|\g_\La|,\Lambda}(\d\eta) P(\d\g)\\
& = \int   \sum_{\xn\subset\eta }f(\xn, (\eta \setminus \xn) \cup \gamma_{\La^c})\, \frac{\e^{-\beta F_\La(\xn,\eta,\g)}}{Z_{\La,\R}(\g)}\,\B_{|\g_\La|,\Lambda}(\d\eta) P(\d\g)\\
& = \int\B_{n,\Lambda}(\d\xn) \int \frac{\1_{|\gamma_\La|\geq n} \,|\gamma_\La|!}{(|\gamma_\La|-n)!} f(\xn, \tilde\eta \cup \gamma_{\La^c})\, \frac{\e^{-\beta F_\La(\xn,\tilde\eta\,\cup\,\xn,\g)}}{Z_{\La,\R}(\g)}\,\B_{|\g_\La|-n,\Lambda}(\d\tilde\eta) P(\d\g)
\end{align*}
where we made the change of variables $\tilde\eta:=\eta\setminus\xn.$ Since one can check by direct computation that the map
$$
\xn\mapsto F_\La(\xn,\tilde\eta\cup\xn,\g)=: F_\La(\tilde\eta,\g),
$$
is constant, we obtain the factorization
\eq
\label{facto}
\tCP^{(n)}(f)=\int  f(\xn,\zeta) \, \Leb^{\otimes n}(\d\xn)\, \Q_n^{\La}(\d \zeta)
\qe
where the measure $ \Q_n^{\La}$ is defined for any measurable map $g:\Conf(\R)\to[0,\infty)$ by,
$$
\int g(\zeta)\Q_n^{\La}(\d \zeta):=\int \frac{\1_{|\gamma_\La|\geq n} \,|\gamma_\La|!}{|\La|^n(|\gamma_\La|-n)!} g(\tilde\eta \cup \gamma_{\La^c})\, \frac{\e^{-\beta F_\La(\tilde\eta,\g)}}{Z_{\La,\R}(\g)}\,\B_{|\g_\La|-n,\Lambda}(\d\tilde\eta) P(\d\g).
$$
Finally, since \eqref{facto} holds true after replacing $\La$ by any bounded Borel set $\La'\supset\La$, we obtain by taking $f:=\1_{\La^n}\otimes g$, 
$$
\int g(\zeta)\Q_n^{\La}(\d \zeta)=\int g(\zeta)\Q_n^{\La'}(\d \zeta)
$$
for any measurable map $g:\Conf(\R)\to[0,\infty)$, and thus $\Q_n:=\Q_n^\La$ does not depend on $\La$.

% This can be checked from the definitions, but the following informal computation provides the intuition for it: let us write 
%$$
%\H_{\La}(\eta)+\M_{\La,\R}(\eta, \g)  =  \frac 12 \eta^{\times 2} +   \eta \times \gamma_{\La^c},
%$$
%and expand it, writing $\eta = \left(\eta - \Typical_n\right) + \Typical_n$ as 
%\begin{multline*}
%\frac 12 \eta^{\times 2} +   \eta \times \gamma_{\La^c} = \frac 12 (\eta - \Typical_n)^{\times 2} + \frac 12  \Typical_n^{\times 2} +   (\eta - \Typical_n) \times \Typical_n \\ +   (\eta - \Typical_n) \times \gamma_{\La^c} +   \Typical_n \times \gamma_{\La^c}.
%\end{multline*}
%On the other hand, let us write, as in \eqref{costheuristic}
%\begin{multline*}
%\h\left(\Typical_n, \left(\eta \setminus \Typical_n \cup \g_{\La^c} \right)\right)  = \frac 12 \Typical_n^{\times 2} \\  +   (\Typical_n - \Reference_n) \times \left( \eta - \Typical_n + \g_{\La^c} \right),
%\end{multline*}
%Substracting one equality from the other, we are left with
%\begin{multline*}
%\H_{\La}(\eta)+\M_{\La,\R}(\eta, \g) - \h(\xn,\eta\cup \g_{\La^c} \setminus\xn) \\
%= \frac 12 (\eta - \Typical_n)^{\times 2} +   (\eta - \Typical_n) \times \gamma_{\La^c} +   \Reference_n \times \left( \left(\eta - \Typical_n\right) + \gamma_{\La^c} \right), 
%\end{multline*}
%which only depends on $\eta - \Typical_n$ (namely, $\eta  \setminus\xn$) and $\gamma_{\La^c}$. Let us also observe that the Bernoulli process with $|\gamma_\La|$ points can be split into the product of two Bernoulli processes, up to some combinatorial factor.

\end{proof}

\subsection{Rigidity and ergodicity}
\label{sec:rigidergo}

We now provide a convenient characterization of number-rigidity. First, in the next lemma we show that, when one wants to prove number-rigidity, it is enough to restrict to an increasing countable family instead of all bounded measurable subsets $\La\subset\R$.

\begin{lemma} 
\label{RigidityCountable}Given any countable family $(B_m)_{m\geq 1}$ of  bounded Borel subsets  of $\R$ satisfying $\cup_m B_m=\R$,  $P$ is number-rigid if and only if, for any $m\geq 1$,  there exists a measurable function $\Number_{B_m}:\Conf(B_m^c)\to\N$ such that $|\gamma_{B_m}|=\Number_{B_m}(\g_{B_m^c})$ $P$-a.s.
\end{lemma}

\begin{proof} Let $\La\subset\R$ be measurable and bounded and $m\geq 1$ such that $\La\subset B_m$. If one assumes that $|\gamma_{B_m}|$ is a measurable function of $\gamma_{B_m^c}$, then by writing $|\g_\La|=|\g_{B_m}|-|\g_{B_m\setminus\La}|$ we see that $|\g_\La|$ coincides almost surely with a measurable function of $\gamma_{\La^c}$.
\end{proof}

Next, we need a few definitions. 

\begin{definition} For any bounded Borel set $\Lambda$ and for a fixed $\gamma$, we introduce the event 
$$
\Ext^{\g}_{\La} := \big\{\eta \in \Conf(\R):\; \eta_{\La^c} = \g_{\La^c} \big\}
$$
of having the same exterior configuration than $\gamma$ outside of $\La$. We also need the event 
$$
\Aa_{k,\La}:=\big\{\eta \in \Conf(\R):\;  |\eta_\La|=k\big\}
$$
of having $k\geq 0$ points in $\La$ and, given a point process $P$ on~$\R$ and $n\geq 1$,
\begin{multline}
\label{def:NoRigidrm}
\NoRigid^P_{n, \La} := \Big\lbrace \g \in \Conf(\R)  :\; \exists k\in\N,\\
 P\left(\Aa_{k,\La} \Big| \Ext^{\g}_{\La} \right) > 0\; \text{ and }\; P\left(\Aa_{k+n,\La} \Big| \Ext^{\g}_{\La} \right) > 0 \Big\rbrace.
\end{multline} 
The latter can be informally understood as the set of configurations $\g$ for which, with positive probability under $P$, it is possible to generate two configurations with a number of points in $\La$ which differs by $n$, conditionally on having the same exterior configuration given by $\g_{\La^c}$. Formally, since we may have  $P(\Ext^{\g}_{\La})=0,$ one has to proceed more carefully. Let us define $\mathscr F_{\La^c}$ as the $\sigma$-algebra 
generated by the random variable $\g\mapsto \gamma_{\La^c}.$
Consider the conditional expectation $\E_P[ \1_{\Aa_{k,\La}  } | \mathscr F_{\La^c}]$, which is a $\mathscr F_{\La^c}$-measurable random variable, and thus can be seen as a measurable function from $\Conf(\La^c)$ to $[0,1]$.  We set 
$$P\left(\Aa_{k,\La} \Big| \Ext^{\g}_{\La} \right):=\E_P[ \1_{\Aa_{k,\La}  } | \mathscr F_{\La^c}](\g_{\La^c}).$$
Finally, recalling that $\La_m:=[-\tfrac m2,\tfrac m2]$, we consider the events
$$
\NoRigid^P_n := \bigcup_{m\ge 1}  \NoRigid^P_{n, \La_m}
$$
and the set of \emph{non-rigid configurations}
\begin{equation}
\label{def:confstar}
\NoRigid^P_{*} := \bigcup_{n\ge 1} \NoRigid^P_n.
\end{equation}
\end{definition}

\begin{proposition}
\label{caracterisationRigidite} A point process $P$ on $\R$ is number-rigid if and only if \linebreak $P(\NoRigid^P_*)=~0$. 
\end{proposition}
\begin{proof}

First, let us assume that $P$ is number-rigid, namely that for any bounded Borel set $\La\subset\R,$ there exists  a measurable function $\Number_\La :\Conf(\La^c)\to\N$ such that  
$$
|\gamma_\La|= \Number_\La(\gamma_{\La^c})\quad \text{ for } P\text{-a.e. }  \gamma.
$$
Then for any $m \ge 1$, for $P$-a.e. $\g$ and $k \neq \Number_{\La_m}(\g_{\La_m^c})$,
$$
P\left(\Aa_{k,\La_m} \big| \Ext^{\g}_{\La_m} \right)=0,
$$ 
and $P(\NoRigid^P_{n, \La_m})=0$ for any $ n    \geq 1$, from which $P(\NoRigid^P_*)=0$ follows. 

Conversely, let us assume that $P$ is \textit{not} number-rigid. Then by Lemma~\ref{RigidityCountable} there exists an integer $m\geq 1$, two integers $0\le k_1<k_2$, and an event $E$ such that $P(E)>0$ and, for any $\g \in E$,
 $$
P\left(\Aa_{k_1, \La_m} \big| \Ext^{\g}_{\La_m} \right)>0, \qquad  P\left(\Aa_{k_2, \La_m} \big| \Ext^{\g}_{\La_m} \right)>0. 
$$
%Let $m\ge 1$ be an integer such  that $\La \subset \La_m$. There exists at least one $k \ge 0$ such that $P(E \cap \Aa_{k, \La_m \setminus \La})>0$ and, for $P$-a.e. $\g\in E \cap \Aa_{k, \La_m \setminus \La}$, we have
%\begin{align*}
%P\left(\Aa_{k+k_1, \La_m}\big|\Ext^{\g}_{\La_m}\right)& \geq P\left(\Aa_{k+k_1, \La_m}\big|\Ext^{\g}_{\La}\right)\\
%& \geq P\left(\Aa_{k_1, \La}\big|\Ext^{\g}_{\La}\right)>0 
%\end{align*}
%and similarly,
%$$
% P\left(\Aa_{k + k_2, \La_m} \big| \Ext^{\g}_{\La_m} \right)>0. 
%$$
Thus, if $n:=k_2-k_1$,  we have $P(\NoRigid^P_{n, \La_m})>0$ and hence $P(\NoRigid^P_*)>0$.
\end{proof}

We next make use in a crucial way of the notion of ergodic point processes (which are in this work always assumed to be stationary). An event $E \subset \Conf(\R)$  is $P$-a.s. translation-invariant, if  for $P$-a.e. $\gamma\in\Conf(\R)$ and  every $x\in\R$, $\1_E(\gamma)=\1_E (\gamma+x).$

\begin{definition}[Ergodic point process]
We say that a  point process $P$ on $\R$ is \textit{ergodic} if it is stationary and, for any  $P$-a.s. translation-invariant event $E$, we have $P(E)\in\{0,1\}$.
%$\circ \tau_x$ are $P$-almost surely equal) 
\end{definition}

Proposition~\ref{caracterisationRigidite} has the following consequence for ergodic point processes.

\begin{corollary}\label{nonrigide ergodic}
Let $P$ be an ergodic point process on $\R$ which is \textit{not} number-rigid. Then there exists $n \ge 1$ such that $P(\NoRigid^P_n)=1$ and furthermore
\begin{equation} 
\label{proprigidite} 
\lim_{m \to \infty}  P(\NoRigid^P_{n, \La_m})=1.
\end{equation}
\end{corollary}

\begin{proof}
If $P$ is an ergodic point process on $\R$ which is not number-rigid, Proposition \ref{caracterisationRigidite}, yields $P(\NoRigid^P_*)>0$ and therefore there exists $n \ge 1$ such that $P(\NoRigid^P_n)>0$. Since the event $\NoRigid^P_n$ is $P$-a.s. translation-invariant this implies that $P(\NoRigid^P_n)=1$. Moreover, \eqref{proprigidite} follows by monotone convergence.
\end{proof}

The interest we have in ergodic processes comes from the following decomposition. 

\begin{proposition} 
\label{splitErgo}
Let $P$ be a  point process on $\R$ which is stationary and satisfies \ref{canDLR}. Let $\mathscr I$ be the $\sigma$-algebra of the translation-invariant Borel sets of $\Conf(\R)$. Then there exists a family of point processes $(P_\eta)_{\eta\in\Conf(\R)}$ such that for $P$-a.e. $\eta$, the point process $P_\eta$ is ergodic, satisfies \ref{canDLR} and  
\eq
%\label{Peta}
P_\eta(\cdot):=\E_P[\bs 1_{\cdot}|\mathscr I](\eta).
\qe
In particular we have the standard decomposition of $P$ via its Gibbsian ergodic phases
\eq
\label{mixingP}
P=\int P_\eta\; P(\d\eta).
\qe
\end{proposition}

\begin{proof}
 Since $\Conf(\R)$ is a Polish space, there exists a regular conditional probability with respect to $\mathscr I$ (see \cite[Theorem 10.2.2]{Dud02}), namely 
there exists a version of the conditional expectation $A \mapsto \E_P[\bs 1_{A}|\mathscr I]$ such that, for any $\eta \in \Conf(\R)$,
\eq
%\label{Peta}
A \mapsto P_\eta(A):=\E_P[\bs 1_{A}|\mathscr I](\eta)
\qe
defines a probability measure on $\Conf(\R)$  and we  have
$$
P(A)=\E_P\big[\,\E_P[\1_{A} | \mathscr I]\,\big]=\int P_\eta(A)\, P(\d\eta).
$$
Note that $P_\eta$ is a stationary point process for $P$-a.e. $\eta$ since $P$ is a stationary point process and by definition of $\mathscr I$. Moreover, the quantity $\E_{P_\eta}[\mathbb W(\g)]$ is necessarily finite for $P$-a.e. $\eta$ since otherwise $\E_{P}[\mathbb W(\g)]$ would be infinite, and since for any $A\in\mathscr I$ we have
$$
P_\eta(A)=\E_P[\bs 1_A|\mathscr I](\eta)=\1_A(\eta)\in\{0,1\},
$$
we see that $P_\eta$ is ergodic. 
Moreover, let $\overline{\mathscr F_{\La^c}}$ be the $\sigma$-algebra generated by the random variables $\g\mapsto \gamma_{\La^c}$ and $\g\mapsto|\gamma_\La|$, and let $\overline{\mathscr F_{\infty}}:=\cap_{m\geq 1}\overline{\mathscr F_{\La_m^c}}$. 

By the ergodic Theorem, for any local event $A$ and for $P$-a.e. $\eta$

\begin{equation} \label{theoergo}
\E_P[\bs 1_A|\mathscr I](\eta)= \lim_{m\to\infty} \frac{1}{|\La_m|} \sum_{u\in \La_m\cap \mathbb{Z}} \bs 1_A(\eta+u).
\end{equation} 

So $\E_P[\bs 1_A|\mathscr I]$ is in fact $\mathscr I\cap \overline{\mathscr F_{\infty}}$-measurable and for $P$-a.e. $\eta$,  $P_\eta(.)=\E_P[\bs 1_.|\mathscr I\cap \overline{\mathscr F_{\infty}} ]$.

Since $P$ satisfies \ref{canDLR}, \cite[Theorem 2.2]{Preston} states that  there exists a version of $A\mapsto \E_P[\1_A|\overline{\mathscr F_\infty}],$ such that, for any $\xi \in \Conf(\R),$  $\E_P[\1_\cdot|\overline{\mathscr F_\infty}](\xi)$ is a point process on $\R$ that satisfies \ref{canDLR}.  By  writing 
$$
P_\eta (A)= \E_P[\, \E_P[\,\1_A|\overline{\mathscr F_\infty}\,] |\mathscr I\cap \overline{\mathscr F_{\infty}}](\eta)=\int\E_P[\,\1_A|\overline{\mathscr F_\infty}\,](\xi) \,P_\eta(\d\xi)
$$
we see that $P_\eta$ satisfies \ref{canDLR}, and thus $P$ can be written as a mixture of ergodic probability measures satisfying \ref{canDLR}.\end{proof}

Combined with the previous proposition, the next result will allow us to restrict to ergodic processes to prove Theorem~\ref{Theo-rigidity}.
\def \Flm{\mathscr{F}_{\La_m^c}}
\def \sI{\mathscr{I}}
\begin{proposition}
\label{rigiditemixing} 
Let us write $P = \int P_{\eta}\, P(\d \eta)$ as in Proposition \ref{splitErgo}. If $P_\eta$ is number-rigid for $P$-a.e. $\eta$, then $P$ is number-rigid.
\end{proposition}
\begin{proof}
In view of Lemma~\ref{RigidityCountable}, it is enough to show that, for all $m \geq 1$, the equality
$$
|\g_{\La_m}|=\E_{P}\left[\,|\g_{\La_m}|\,|\mathscr F_{\La_m^c}\right],
$$
holds $P$-almost surely.

By assumption, for $P$-a.e. $\eta$, the process $P_{\eta}$ is number-rigid, hence we have
$$
|\g_{\La_m}|=\E_{P_{\eta}}\left[\,|\g_{\La_m}|\,|\mathscr F_{\La_m^c}\right], \text{ $P_{\eta}$-a.s.},
$$
thus it suffices to show
\begin{equation}
\label{identiteconditionnelle}
\E_P \left[ |\g_{\La_m}| \Big| \Flm \right] = \E_{P_{\eta}} \left[ |\g_{\La_m}| \Big| \Flm \right], \text{ $P_{\eta}$-a.s}.
\end{equation}
Both sides of \eqref{identiteconditionnelle} are $\Flm$-measurable random variables. Let $f$ be a bounded, $\Flm$-measurable random variable. We may write
$$
\E_{P_{\eta}} \left[ f \E_{P} \left[ |\g_{\La_m}| \Big| \Flm \right] \right] = \E_{P_{\eta}} \left[ \E_{P} \left[ f |\g_{\La_m}| \Big| \Flm \right] \right],
$$
and by definition of $P_{\eta}$ and \eqref{theoergo} we have, % 
$$
\E_{P_{\eta}} \left[ \E_{P} \left[ f |\g_{\La_m}| \Big| \Flm \right] \right]  = \E_P \left[ \E_{P} \left[ f |\g_{\La_m}| \Big| \Flm \right] \Bigg| \sI \cap  \Flm\right](\eta),
$$
where $\sI$ is as in Proposition \ref{splitErgo}. By the “tower property” of conditional expectation, we obtain
\begin{equation}
\label{cote1}
\E_{P_{\eta}} \left[ f \E_{P} \left[ |\g_{\La_m}| \Big| \Flm \right] \right] =  \E_{P} \left[ f |\g_{\La_m}| \Big| \sI \cap  \Flm \right](\eta),
\end{equation}
but by definition of $P_{\eta}$ we have
\begin{equation}
\label{cote2}
\E_{P} \left[ f |\g_{\La_m}| \Big| \sI \cap  \Flm \right](\eta) = \E_{P_{\eta}} \left[ f |\g_{\La_m}|  \right].
\end{equation}
Combining \eqref{cote1} and \eqref{cote2} we see that
$$
\E_{P_{\eta}} \left[ f \E_{P} \left[ |\g_{\La_m}| \Big| \Flm \right] \right] = \E_{P_{\eta}} \left[ f |\g_{\La_m}|  \right] = \E_{P_{\eta}} \left[ f \E_{P_{\eta}} \left[ |\g_{\La_m}| \Big| \Flm \right]  \right],
$$
where the last equality is simply the definition of a conditional expectation. Since this is true for any $f$ bounded and $\Flm$-measurable, we get \eqref{identiteconditionnelle}, which concludes the proof.
\end{proof}

\subsection{Ergodic solutions of DLR equations which are not number-rigid }
Theorem \ref{representationCPn} gave a description of the Campbell measures of stationary solutions of  \ref{canDLR}. In this section we further assume these solutions are ergodic and not number-rigid and improve on the previous description. Let us recall that the main rigidity result we have in mind, as stated in Theorem \ref{Theo-rigidity}, claims that solutions of the DLR equations are \textit{all} rigid, so the result of the present section should turn out to be empty; the next theorem is the main part of our proof by contradiction.
 
\begin{theorem}
\label{Theoabsurde}
Let $P$ be an ergodic point process on $\R$ satisfying \ref{canDLR} and assume that $P$ is not number-rigid. Then there exists $n \ge 1$ and a measurable function $\create_n:\Conf(\R)\to[0,\infty)$ such that the Campbell measure $\CP^{(n)}$ is absolutely continuous with respect to $\Leb^{\otimes n} \otimes P$ and with density
\begin{equation}\label{GrandCandescription}
 \frac{\d \CP^{(n)}}{\d\Leb^{\otimes n}\otimes P} (\xn,\g) = \create_n(\g) \,\e^{-\beta \h(\xn,\g)},
 \end{equation}
 where $\h$ is defined in \eqref{costxg}.
\end{theorem}
Compared to Theorem \ref{representationCPn}, the important change is that we gained information on the second marginal of $\CP^{(n)}$: Instead of being some abstract measure $\Q_n$ as in \eqref{CPnrep1}, here we see that it is absolutely continuous with respect to $P$ itself and obtain crucial information on its density.

%\AH{Before giving the proof of equation \eqref{GrandCandescription}, let us provide some intuition. Going back to Remark \ref{rem:GNZ}, we know that we can try to compute $\CP^{(n)}$ by choosing a “reference” for the $n$-tuple of points, and to consider the quotient
%$$
%\frac{\CP^{(n)}(\Typical_n, \Neighb)}{\CP^{(n)}(\Reference_n, \Neighb)}.
%$$
%This corresponds to a certain energetic cost, as explained in Theorem \ref{representationCPn}. If $P$ is rigid, we cannot say much more because the presence of the $n$-tuple of points is important: if it is likely (i.e. the density is not zero) to observe a certain $n$-tuple of points $\Reference_n$ and the $\Neighb$, then it is impossible to observe e.g. only  the $\Neighb$, because this would amount to removing $n$ points. However, if $P$ is “very” non-rigid (and a non-rigid ergodic process is “very” non-rigid because $P(\NoRigid_n)=1$ if we choose $n$ correctly according to Corollary \ref{nonrigide ergodic}), we can remove these $n$ reference points and evaluate
%$$
%\frac{\CP^{(n)}(\Reference_n, \Neighb)}{dP(\Neighb)}, 
%$$
%which corresponds to $\create_n(\gamma)$ as in \eqref{GrandCandescription}. One can think of this term as “the cost of creating a $n$-tuple of points $\Reference_n$ inside the configuration $\gamma$”.
%}

\begin{proof}
Let $P$ be an ergodic point process on $\R$ satisfying  \ref{canDLR} which is not number-rigid. Corollary \ref{nonrigide ergodic} provides $n \geq 1$ such that  $P(\NoRigid_n^P)=1$. Recalling Theorem \ref{representationCPn} it is enough to prove that the measure $\Q_n$ of \eqref{CPnrep1} is absolutely continuous with respect to $P$. To do so, let $\Eve$ be an event such that $P(\Eve) = 0$ and prove that $\Q_n(\Eve) = 0$, which is enough to prove the existence of $\create_n$ according to the Radon-Nikodym theorem.

Since  $\h(\xn,\gamma)$ is finite for $P$-a.e. $\gamma$ and $\xn\in\R^n$,  the density \eqref{CPnrep1} is positive and it is enough to show that 
$
\CP^{(n)}([0,1]^n \times \Eve) = 0.
$
By definition of Campbell measures, we have
$$
\CP^{(n)}([0,1]^n \times \Eve) = \int \sum_{\xn\subset\gamma}
\1_{[0,1]^n}(\xn)\, \1_{\Eve}(\g\setminus\xn)P(\d\g).
$$
Since $P(\NoRigid_n)=1$, by monotonicity we also have $\lim_{m \to \infty} P(\NoRigid_{n,\La_m})=1$ and thus, by monotone convergence:
\begin{equation*}
 \CP^{(n)}([0,1]^n \times \Eve) = \lim_{m\to \infty} \int  \1_{\NoRigid_{n,\La_m}} (\g) \sum_{\xn\subset\gamma}\1_{[0,1]^n}(\xn)\, \1_{\Eve} (\g\setminus\xn)P(\d\g).
\end{equation*}
Recalling  \eqref{def:NoRigidrm}, for any integers $n,k,$ we define
\begin{multline}
\label{def:NoRigidrmk}
A^P_{n,k, \La} := \Big\lbrace \g \in \Conf(\R)  :\; \exists\, p \le k,\\
 P\left(\Aa_{p,\La} \Big| \Ext^{\g}_{\La} \right) > 0\; \text{ and }\; P\left(\Aa_{p+n,\La} \Big| \Ext^{\g}_{\La} \right) > 0 \Big\rbrace
\end{multline} 
and 
$$ \NoRigid_{n,k,\La_m} := A^P_{n,k, \La} \setminus A^P_{n,k-1, \La},$$
so that one can write the disjoint union:  $\NoRigid_{n,\La_m} = \cup_k  \NoRigid_{n,k,\La_m}.$
Now, if we set:
$$
\label{Termkm}
\Term_{k,m} :=    \int \1_{\Aa_{k, \La_m}}(\gamma) \1_{\NoRigid_{n,k,\La_m}} (\g)   \sum_{\xn\subset\gamma} \1_{[0,1]^n}(\xn)\, \1_{\Eve} (\g\setminus\xn)P(\d\g),
$$
we obtain
\eq
\label{term3}
\CP^{(n)}([0,1]^n \times \Eve) 
\le \lim_{m \to \infty} \sum_{k=0}^{\infty}\Term_{k,m} .
\qe
We introduce for convenience
$$
\Dens^{\La_m}_{\gamma_{\La_m^c}, k}(\eta) := \frac{1}{ Z_{\La_m, \R}( \g)}\, \e^{- \beta(\H_{\La_m}(\eta)+ \M_{\La_m, \R}(\eta, \g))} \,\bs1_{|\eta|=k}
$$
the density with respect to $\B_{k,\La_m}$ of $P$  conditionally to $\gamma_{\La_m^c}$ and having $k$ points in $\La_m$,  provided it makes sense. Recalling Remark~\ref{movedontmove},  the right hand side does only depend on $\gamma_{\La_m}$ through its cardinality, which explains the notation. 
 By using \ref{canDLR}  we obtain
\begin{multline*}
\Term_{k,m} = \int  \1_{\NoRigid_{n, k,\La_m}}(\eta\cup\g_{\La_m^c}) \\ \times   \sum_{\xn\subset\eta} \1_{ [0,1]^n}(\xn)  \1_{\Eve}(\eta \cup \g_{\La_m^c} \setminus\xn)\Dens^{\La_m}_{\gamma_{\La_m^c}, k}(\eta) \B_{k,\Lambda_m}(\d\eta) P(\d\g).
\end{multline*}
%The Bernoulli point process with $k$ points in $\La_m$ can be split into the product of two processes with respectively $k-n$ and $n$ points. 
Similarly as in the proof of Theorem~\ref{representationCPn}, we set $\tilde\eta:=\eta\setminus\xn$ and obtain,
\begin{multline}
\label{term}
\Term_{k,m} = \frac{k!}{m^n(k-n)! }    \int  \1_{\NoRigid_{n, k,\La_m}}(\tilde\eta\cup\xn\cup\g_{\La_m^c}) 
 \1_{\Eve}(\tilde\eta \cup \g_{\La_m^c}) \\
 \qquad \times\Dens^{\La_m}_{\gamma_{\La_m^c}, k}(\tilde\eta\cup\xn) \B_{n,[0,1]}(\d\xn) \B_{k-n,\Lambda_m}(\d\tilde\eta)P(\d\g)\\
 = \frac{k!}{m^n(k-n)! }    \int  \1_{\NoRigid_{n, k,\La_m}}(\tilde\eta\cup\xn\cup\g_{\La_m^c}) 
 \1_{\Eve}(\tilde\eta \cup \g_{\La_m^c})\Dens^{\La_m}_{\gamma_{\La_m^c}, k-n}(\tilde\eta) \\
 \qquad \times \left(\int \frac{\Dens^{\La_m}_{\gamma_{\La_m^c}, k}(\tilde\eta\cup\xn)}{\Dens^{\La_m}_{\gamma_{\La_m^c}, k-n}(\tilde\eta)}\B_{n,[0,1]}(\d\xn) \right) \B_{k-n,\Lambda_m}(\d\tilde\eta)P(\d\g).
\end{multline}
The integrand inside the parentheses is finite since the denominator is  positive  on  the event where $\tilde\eta\cup\xn\cup\gamma_{\La_m^c}\in\NoRigid_{n,k,\La_m}$.
Since \ref{canDLR}  yields for any $k,n,m\geq 0$ and $\xn\in\R^n$,
\begin{align*}
P(E) & = \int \1_{E}(\tilde\eta\cup\gamma_{\La_m^c})\Dens^{\La_m}_{\gamma_{\La_m^c}, |\gamma_{\La_m}|}(\tilde\eta)\,\B_{|\g_{\La_m}|,\La_m}(\d\tilde\eta)P(\d\g)\\
&\geq \int  
 \1_{\NoRigid_{n, k,\La_m}}(\tilde\eta\cup\xn\cup\gamma_{\La_m^c})
 \1_{E}(\tilde\eta\cup\gamma_{\La_m^c}) \Dens^{\La_m}_{\tilde\gamma_{\La_m^c}, k-n}(\eta)\,\B_{k-n,\La_m}(\d\tilde\eta)P(\d\g)
\end{align*}
and $P(E)=0$ by assumption, we obtain with \eqref{term} that $\Term_{k,m} =0$ for every $k,m$. Thus  $\CP^{(n)}([0,1]^n \times \Eve)=0$ and this concludes the proof of Theorem \ref{Theoabsurde}.

\end{proof}

Now let us give some properties of the function $\create_n$. In Lemmas \ref{stationaritygn} and \ref{lem:createbougepoints}, we work under the same assumptions as in Theorem \ref{Theoabsurde}. 

First, we have the following simple result concerning the effect of translations on $\create_n$ appearing in Theorem~ \ref{Theoabsurde}, which is a direct consequence of the stationarity. 

\begin{lemma} Under the assumptions of Theorem~\ref{Theoabsurde},
\label{stationaritygn}
for any $x\in \R$,  $P$-a.e. $\g$ and $\Leb^{\otimes n}$-a.e. $\xn$, we have 
\begin{equation*}
\create_n(\g)\,\e^{-\beta \h(\xn,\g)}
 =  \create_n(\g - x)\,\e^{-\beta \h\left(\xn+x, \g-x \right)}.
\end{equation*}
\end{lemma}

The following property of $\create_n$ will be the crucial for the forthcoming proof by contradiction of number-rigidity.

\begin{lemma} 
 \label{lem:createbougepoints}
 Under the assumptions of Theorem~\ref{Theoabsurde},
for $P$-a.e. $\g$,  $\Leb^{\otimes n}$-a.e. $\bv y_n$, and any $\bv x_n \subset \g$, we have
\begin{equation}
\label{bougepointgn}
\create_n(\g)= \create_n\left(\g\setminus \bv x_n \cup \bv y_n\right) \e^{n\beta  \sum_{j=1}^n g (y_j)- g(x_j) }.
\end{equation}
\end{lemma}

%\AH{If we think of $\create_n(\g)$ as “the cost of creating $n$ points at $0$ inside $\gamma$”, then the identity \eqref{bougepointgn} is natural: if we replace $x$ by $y$ in $\gamma$, the interaction with the $n$ points in $0$ changes as 
%$$
%+ y \times (0, 0 \dots, 0) - x \times (0, \dots, 0),
%$$
%which yields, since the interaction kernel is $- \log | \cdot |$, 
%$$
%- n \log |y| + n \log |x|, 
%$$
%and this energy is then included in the Boltzmann factor $e^{-\beta \bullet }$, hence \eqref{bougepointgn}. For convenience, we will write the exponential term in the right-hand side of \eqref{bougepointgn} as $\left| \frac{x}{y} \right|^{n\beta}$.}

\begin{proof}
It is enough to show that, for $\Q_n$-a.e. $\g$ and $\Leb^{\otimes n}$-a.e. $\bv x_n ,\bv y_n$, we have 
\eq
\label{createbis}
\create_n(\g\cup \bv x_n)\, \e^{n\beta \sum_{j=1}^ng(x_j)}= \create_n\left(\g\cup \bv y_n\right) \,\e^{n\beta \sum_{j=1}^ng(y_j)}.
\qe
Indeed, let $h:\R^n\times \R^n\times \Conf(\R)\to[0,\infty)$ be a measurable test function. We have by  definition of $\CP^{(n)}$, Theorem~\ref{representationCPn} and \eqref{createbis},
\begin{align*}
& \int \sum_{\bv x_n\subset\g} h(\bv x_n,\bv y_n,\g)\,\create_n(\g)\, P(\d\g)\Leb^{\otimes n}(\d \bv y_n) \\
%= & \int \sum_{\bv x_n\subset\g} h(\bv x_n,\bv y_n,\g\cup\bv x_n\setminus \bv x_n)\create_n(\g\cup \bv x_n\setminus \bv x_n) P(\d\g)\Leb^{\otimes n}(\d \bv y_n)  \\
= & \int\CP^{(n)}\Big[(\bv x_n,\gamma)\mapsto h(\bv x_n,\bv y_n,\g\cup \bv x_n)\create_n(\g\cup \bv x_n) \Big]\Leb^{\otimes n}(\d \bv y_n)\\
= & \int  h(\bv x_n,\bv y_n,\g\cup \bv x_n)\create_n(\g\cup \bv x_n)\e^{-\beta\h(\bv x_n,\gamma)}  \Leb^{\otimes n}(\d \bv x_n)\Q_n(\d\gamma)\Leb^{\otimes n}(\d \bv y_n) \\
= & \int h(\bv x_n,\bv y_n,\g\cup \bv x_n)\create_n(\g\cup \bv y_n) \e^{n\beta  \sum_{j=1}^ng(y_j)-g(x_j) }\\
& \hspace*{5.5cm} \e^{-\beta\h(\bv x_n,\gamma)} \Leb^{\otimes n}(\d \bv x_n)\Q_n(\d\gamma)\Leb^{\otimes n}(\d \bv y_n) \\
= & \int\CP^{(n)}\Big[(\bv x_n,\gamma)\mapsto h(\bv x_n,\bv y_n,\g\cup \bv x_n)\create_n(\g\cup \bv y_n)  \e^{n\beta  \sum_{j=1}^ng(y_j)-g(x_j)  }\Big]\Leb^{\otimes n}(\d \bv y_n)\\
= &  \int \sum_{\bv x_n\subset\g} h(\bv x_n,\bv y_n,\g)\,\create_n(\g \cup\bv y_n\setminus\bv x_n)  \e^{n\beta  \sum_{j=1}^ng(y_j)-g(x_j) } \,P(\d\g)\Leb^{\otimes n}(\d \bv y_n),
\end{align*}
from which our claim follows.

Let  $f:\R^{n}\times \R^n\times \Conf(\R)\to[0,\infty)$ be any measurable function, and set 
$$
f_1(\xn,\g):=\sum_{\bv y_n\subset\g } f(\bv x_{n}, \bv y_n,\g\setminus   \bv y_n),\qquad f_2(\bv y_n,\g):=\sum_{\bv x_n\subset\g } f(\bv x_{n},\bv y_n,\g\setminus \bv x_n).
$$
We have (writing $(\bv x_n,\bv y_n)=\bv z_{2n}\in\R^{2n}$),  
\begin{align*}
\CP^{(2n)}(f) & = 
\int \sum_{\bv z_{2n}\subset\g}  f(\bv z_{2n},\g\setminus \bv z_{2n}) P(\d\g) \\
&  =  \int\sum_{
\bv x_n\subset\g
 }  \Big(\sum_{\bv y_n\in\g\setminus \bv x_n } f(\bv x_{n}, \bv y_n,\g\setminus (\bv x_{n}\cup \bv y_n))\Big) P(\d\g)\\
 & = \CP^{(n)}(f_1)
\end{align*}
and similarly $\CP^{(2n)}(f)=\CP^{(n)}(f_2)$. Next, using Theorem~\ref{Theoabsurde} and then Theorem~\ref{representationCPn}, we obtain

\begin{align*}
 \CP^{(n)}(f_1) 
 & =  \int\create_n(\g)\e^{-\beta \h(\bv x_n,\g)}\left[\sum_{\bv y_n\in\g } f(\bv x_{n}, \bv y_n,\g\setminus  \bv y_n) \right] 
 \Leb^{\otimes n}(\d \bv x_{n}) P(\d\g)\\
 %& =  \int \sum_{\bv x_k \subset\g} \create_n(\g\cup\bv x_k\setminus \bv x_k)\e^{-\beta \h(\bv x_n,\g\cup\bv x_k\setminus\bv x_k)}f(\bv x_n\cup\bv x_k,\g\setminus\{\bv x_k\}) P(\d\g) \Leb^{\otimes n}(\d \bv x_{n}) \\
 &=\int \CP^{(n)}\left[ (\bv y_n,\gamma)\mapsto\create_n(\g\cup \bv y_n)\e^{-\beta \h(\bv x_n,\g\cup \bv y_n)}f(\bv x_n, \bv y_n,\g)\right]\Leb^{\otimes n}(\d \bv x_{n})\\
 &= \int  \create_n(\g\cup \bv y_n)\e^{-\beta \h(\bv x_n,\g\cup \bv y_n)-\beta  \h(\bv y_n,\g)}\\
 & \qquad\qquad\qquad\qquad\qquad\quad \times f(\bv x_n, \bv y_n,\g)\Leb^{\otimes n}(\d \bv x_{n})\Leb^{\otimes n} (\d \bv y_n) \Q_n(\d \g).
 \end{align*}
 The same computation with $f_2$ yields
 \begin{multline*}
  \CP^{(n)}(f_2)= \int  \create_n(\g\cup \bv x_n)\e^{-\beta \h(\bv y_n,\g\cup \bv x_n)-\beta  \h(\bv x_n,\g)}\\
  \times f(\bv x_n,\bv y_n,\g)\Leb^{\otimes n}(\d \bv x_{n})\Leb^{\otimes n} (\d \bv y_n)\Q_n(\d \g).
 \end{multline*}
 Thus, since the test function $f$ is arbitrary, we have for $\Q_n$-a.e. $\gamma$ and $\Leb^{\otimes n}$-a.e. $\bv x_n,\bv y_n$,
 $$
 \create_n(\g\cup \bv x_n)\e^{-\beta \h(\bv y_n,\g\cup \bv x_n)-\beta  \h(\bv x_n,\g)}=\create_n(\g\cup \bv y_n)\e^{-\beta \h(\bv x_n,\g\cup \bv y_n)-\beta  \h(\bv y_n,\g)},
 $$
 from which, after using the definition of $\h$ and several simplifications, we obtain \eqref{createbis}.
\end{proof}

\subsection{Proof of Theorem \ref{Theo-rigidity}} \label{sectionconclusionproof}

We are finally in position to provide a proof for Theorem~\ref{Theo-rigidity}. Let $P$ be a stationary point process satisfying \ref{canDLR}.

\subsubsection{Restriction to ergodic processes.}
First,  as a consequence of Proposition~\ref{splitErgo} and Proposition~\ref{rigiditemixing}, we assume without loss of generality that $P$ is further ergodic in the assumptions of  Theorem~\ref{Theo-rigidity} to prove this theorem.

\subsubsection{Strategy of the proof.} The proof goes by contradiction.

Assume that the ergodic point process $P$ is \textbf{not} number-rigid. By Theorem \ref{Theoabsurde}, there exists $n\ge1$ and a function $\create_n$ such that \eqref{GrandCandescription} holds true; from now this $n$ is fixed.  For any $m\geq 1$ and $K\geq 0$, we denote by $\Nn_m^K$ the event
\begin{equation}
\label{def:Nn}
\Nn_m^K := \left\lbrace \gamma \in \Conf(\R): \; |\gamma_{[m, m+1]}| \leq K \right\rbrace
\end{equation}
and consider the test function  $\Test_{m,K}$ defined by
\begin{equation}
\label{TestFunctionmK}
\Test_{m,K}(\xn, \gamma) = \1_{[0,1]^n}(\xn) \1_{\Nn_0^K}(\gamma) \1_{\Nn_m^K}(\gamma)  \sum_{\bv y_n\subset \gamma} \1_{[m,m+1]^n}(\bv y_n).
\end{equation}
In the following $K$ is fixed and always assumed to be larger than $2n$.

We compute, using that $\xn\in[0,1]^n$ and $m\geq 1$,
\begin{align*}
\CP^{(n)}(\Test_{m,K}) & = \int \sum_{\xn\subset \gamma} \1_{[0,1]^n}(\xn) \1_{\Nn_0^K}(\gamma\setminus\xn) \1_{\Nn_m^K}(\gamma\setminus\xn)  \sum_{\bv y_n\subset \gamma\setminus\xn} \1_{[m,m+1]^n}(\bv y_n) P(\d\gamma)\\
& =  \int  \1_{\Nn_0^{K+n}}(\gamma) \1_{\Nn_m^K}(\gamma)  \sum_{\xn\subset \gamma} \1_{[0,1]^n}(\xn) \sum_{\bv y_n\subset \gamma} \1_{[m,m+1]^n}(\bv y_n) P(\d\gamma).
\end{align*}
Thus, forgetting about the truncation $\1_{\Nn_0^{K+n}}(\gamma) \1_{\Nn_m^K}(\gamma )$, we should think of \linebreak $\CP^{(n)}(\Test_{m,K})$ as encoding the correlation between the number of $n$-tuples from $\g$ with law $P$ falling into $[0,1]$ and $[m, m+1]$ respectively.  The main idea to reach the contradiction is that $\CP^{(n)}(\Test_{m,K})$ is bounded from above independently on $m$, since  we have the rough upper bound,
\eq
\label{CPisbounded}
\CP^{(n)}(\Test_{m,K}) \leq  (K+n)^n\times K^n,
\qe
but we will prove that its Ces\'aro series diverges by using the long range of the logarithmic interaction. 

\subsubsection{Manipulations on the test function - Part I}
By using Theorem \ref{Theoabsurde}, introducing phantom variables $\bv z_n$ and then using Lemma \ref{lem:createbougepoints}, we can write
\begin{align*}
& \CP^{(n)}(\Test_{m,K}) \\
= &  \int  \create_n(\gamma) \e^{-\beta \h(\xn, \gamma)}\1_{\Nn_0^K}(\gamma) \1_{\Nn_m^K}(\gamma) \sum_{\bv y_n\subset \gamma} \1_{[m,m+1]^n}(\bv y_n) P(\d\gamma)\Leb_{[0,1]}^{\otimes n}(\d \xn)\\
 =& \int  \create_n(\gamma) \e^{-\beta \h(\xn, \gamma)}\1_{\Nn_0^K}(\gamma) \1_{\Nn_m^K}(\gamma) \\
 & \hspace*{5cm}\sum_{\bv y_n\subset \gamma} \1_{[m,m+1]^n}(\bv y_n) P(\d\gamma)
\Leb_{[0,1]}^{\otimes n}(\d \xn)\Leb_{[1,2]}^{\otimes n}(\d \bv z_n)\\
=&  \int   \e^{-\beta \h(\xn, \gamma)}  \1_{\Nn_0^K}(\gamma) \1_{\Nn_m^K}(\gamma) \sum_{\bv y_n\subset \g} \create_n(\gamma \setminus \bv y_n\cup\bv z_n)  \1_{[m,m+1]^n}(\bv y_n) \e^{n\beta \sum_{j=1}^n g(z_j)-g(y_j) } \\
& \qquad\qquad\qquad\qquad\qquad\qquad\qquad\qquad\times P(\d\gamma) \Leb_{[0,1]}^{\otimes n}(\d \xn)\Leb_{[1,2]}^{\otimes n}(\d \bv z_n).
\end{align*}

Next, we use that for any $\bv z_n\in[1,2]^n$ and $\bv y_n\in[m, m+1]^n$ we have
\eq
\label{echangepointslogm}
\e^{n\beta \sum_{j=1}^n g(z_j)-g(y_j)} \geq \e^{-C g(m)},
\qe
for any $m$ large enough and  $C > 0$ only depending on $n, \beta$. This yields
\begin{multline}
\label{CPnTest1}
\CP^{(n)}(\Test_{m,K})
\geq \e^{-Cg(m)}  \int   \e^{-\beta \h(\xn, \gamma)}  \1_{\Nn_0^K}(\gamma) \1_{\Nn_m^K}(\gamma) \\ \times\sum_{\bv y_n\subset \g} \create_n(\gamma \setminus \bv y_n\cup\bv z_n)  \1_{[m,m+1]^n}(\bv y_n) 
P(\d\gamma) \Leb_{[0,1]}^{\otimes n}(\d \xn)\Leb_{[1,2]}^{\otimes n}(\d \bv z_n)
\end{multline}
for any $m$ sufficiently large.

\subsubsection{Manipulations on the test function - Part II}
The integrand in \eqref{CPnTest1} can also be interpreted in terms of a Campbell measure, and using again Theorem \ref{Theoabsurde}, we have for any fixed $\bv x_n,\bv z_n$,
\begin{align}
\label{CPnTest2}
& \int   \e^{-\beta \h(\xn, \gamma)}  \1_{\Nn_0^K}(\gamma) \1_{\Nn_m^K}(\gamma) \sum_{\bv y_n\subset \g} \create_n(\gamma \setminus \bv y_n\cup\bv z_n)  \1_{[m,m+1]^n}(\bv y_n) P(\d\gamma) \nonumber\\
& = \int \e^{-\beta \h(\bv x_n, \gamma \cup \bv y_n)-\beta \h(\bv y_n, \gamma)}  \1_{\Nn_0^K}(\gamma \cup \bv y_n) \1_{\Nn_m^K}(\gamma \cup\bv y_n) \create_n(\gamma \cup\bv z_n)
\create_n(\gamma) \nonumber\\
& \qquad \times P(\d\gamma) \Leb_{[m,m+1]}^{\otimes n}(\d \bv y_n)\nonumber\\
& = \int \e^{-\beta \h(\bv x_n, \gamma \cup \bv y_n)-\beta \h(\bv y_n, \gamma)}  \1_{\Nn_0^K}(\gamma) \1_{\Nn_m^{K-n}}(\gamma) \create_n(\gamma \cup\bv z_n)
\create_n(\gamma) \\
& \qquad \times P(\d\gamma) \Leb_{[m,m+1]}^{\otimes n}(\d \bv y_n).\nonumber
\end{align}
Moreover, recalling Definition \ref{definitionhn}, we have for any $\bv x_n\in[0,1]^n$ and $\bv y_n\in[m, m+1]^n$,
\eq
\label{secondestimatecost}
\big|\h(\bv x_n, \gamma \cup \bv y_n)  - \h(\bv x_n, \gamma)\big|=  \left| \sum_{i,j= 1}^n    g (y_j - x_i)-g(y_j)\right| \leq c
\qe
for some $c>0$ depending only on $n, \beta$. Thus, together with \eqref{CPnTest1}--\eqref{CPnTest2}, we obtain
\begin{multline}
\label{CPnTest4}
\CP^{(n)}(\Test_{m,K})
\geq \e^{-(Cg(m) + c)}  \int  \e^{-\beta \h(\bv x_n, \gamma)-\beta \h(\bv y_n, \gamma)}  \1_{\Nn_0^K}(\gamma) \1_{\Nn_m^{K-n}}(\gamma) \\
 \times \create_n(\gamma \cup\bv z_n)\create_n(\gamma) P(\d\gamma) \Leb_{[0,1]}^{\otimes n}(\d \xn)\Leb_{[m,m+1]}^{\otimes n}(\d \bv y_n)\Leb_{[1,2]}^{\otimes n}(\d \bv z_n).
\end{multline}

Finally, in order to isolate the dependence in $m$ in the remaining integral, let us set 
\eq
\label{IIdef}
\II_{m, K}(\g):=\int  \create_n(\gamma) \e^{-\beta \h(\bv y_n, \gamma)}  \1_{\Nn_m^{K-n}}(\gamma) \Leb_{[m,m+1]}^{\otimes n}(\d \bv y_n)
\qe
so that the previous inequality reads
\begin{multline}
\label{CPnTest5}
\CP^{(n)}(\Test_{m,K})\geq \e^{-(C g(m) + c)}  \int  \II_{m, K}(\g)\,\e^{-\beta \h(\bv x_n, \gamma)}  \1_{\Nn_0^K}(\gamma) 
  \create_n(\gamma \cup\bv z_n) \\ \times P(\d\gamma) \Leb_{[0,1]}^{\otimes n}(\d \xn)\Leb_{[1,2]}^{\otimes n}(\d \bv z_n).
\end{multline}

\subsubsection{An ergodic argument}
We now make change the change of variables  $\bv y_n\mapsto \bv y_n+m$ in \eqref{IIdef} to obtain from Lemma~\ref{stationaritygn},
\begin{align*}
\II_{m, K}(\g)& =\int  \create_n(\gamma)\e^{-\beta \h(\bv y_n+m, \gamma)}  \1_{\Nn_m^{K-n}}(\gamma) \Leb_{[0,1]}^{\otimes n}(\d \bv y_n)\\
& =\int  \create_n(\gamma-m)\e^{-\beta \h(\bv y_n, \gamma-m)}  \1_{\Nn_0^{K-n}}(\gamma-m)\Leb_{[0,1]}^{\otimes n}(\d \bv y_n) \\
& = \II_{0, K}(\g-m).
\end{align*}
Since $P$ is ergodic by assumption, the ergodic theorem implies that  $P$-a.s.
$$
\lim_{M \to \infty} \frac{1}{M} \sum_{m=1}^M\II_{m, K}(\g)= \E_P\big[\II_{0, K}\big].
$$
Note that $\E_P\big[\II_{0, K}\big]>0$, since
\begin{align}
\label{II=camp}
\E_P\big[\II_{0, K}\big]& = \int\create_n(\eta) \e^{-\beta \h(\bv y_n, \eta)} \1_{\Nn_0^{K-n}}(\eta) \Leb_{[0,1]}^{\otimes n}(\d \bv y_n) P(\d \eta)\nonumber\\
& = \CP^{(n)} \left( \1_{[0,1]^n}(\bv y_n) \1_{\Nn_0^{K-n}}(\gamma) \right)\nonumber\\
&= \int \sum_{\bv y_n\subset\g}  \1_{[0,1]^n}(\bv y_n)\1_{\Nn_0^{K}}(\gamma\setminus\bv y_n) P(\d\gamma)\\
& \geq P(|\g_{[0,1]}|=n) >0\nonumber
\end{align}
by Corollary~\ref{DLRcharge}; we used that $K\geq n$ by assumption. Fatou's lemma then yields
\begin{multline}
\label{apresFatou}
\liminf_{M \to \infty} \frac{1}{M} \sum_{m=1}^M \CP^{(n)}(\Test_{m,K}) 
\geq \left(\liminf_{M \to \infty} \frac{1}{M} \sum_{m=1}^M \e^{-(C g(m) + c)} \right) \E_P\big[\II_{0, K}\big]  \\
\times  \int  
\e^{-\beta \h(\bv x_n, \gamma)} \create_n(\gamma \cup\bv z_n) \1_{\Nn_0^K}(\gamma) P(\d\gamma) 
 \Leb_{[0,1]}^{\otimes n}(\d\bv x_n) \Leb_{[1,2]}^{\otimes n}(\bv z_n).
\end{multline}

\subsubsection{The remaining integral and conclusion}
If the integral  in \eqref{apresFatou} vanishes, then 
\eq
\label{termerestant}  
\e^{-\beta \h(\bv x_n, \gamma)} \create_n(\gamma \cup\bv z_n) \1_{\Nn_0^K}(\gamma) =0,
\qe
$P \otimes  \Leb_{[0,1]}^{\otimes n} \otimes \Leb_{[1,2]}^{\otimes n}$-almost everywhere. 
Let us fix $m_0\geq 1$. Note that when we proved the lower bound \eqref{CPnTest4}, the only inequalities we have used were \eqref{echangepointslogm} and \eqref{secondestimatecost}. Since the converse inequality also holds in \eqref{echangepointslogm} after changing the constant $C$ to another positive constant, the exact same line of arguments yields $\kappa>0$ only depending on $m_0,n,\beta$ such that
\begin{multline}
\label{CPupperbound}
\CP^{(n)}(\Test_{m_0,K})\leq \kappa  \int  \II_{m_0, K}(\g)\;\e^{-\beta \h(\bv x_n, \gamma)}  
  \create_n(\gamma \cup\bv z_n)  \1_{\Nn_0^K}(\gamma)\\ \times P(\d\gamma) \Leb_{[0,1]}^{\otimes n}(\d \xn)\Leb_{[1,2]}^{\otimes n}(\d \bv z_n).
  \end{multline}
 Recalling \eqref{II=camp}, we have $\E_P[\II_{m_0, K}]=\E_P[\II_{0, K}]\leq (K+n)^n<\infty$ and thus \eqref{termerestant}--\eqref{CPupperbound}   imply together that $\CP^{(n)}(\Test_{m_0,K})=0$. However, we have the rough lower bound
$$
\CP^{(n)}(\Test_{m,K}) \geq P\left( |\gamma_{[0,1]}| = n \cap |\gamma_{[m, m+1]}| = n \right),
$$
 as soon as $K\geq 2n$, and Corollary~\ref{DLRcharge} then yields that $\CP^{(n)}(\Test_{m_0,K})>0$.
  
  As a consequence, the remaining integral in \eqref{apresFatou} is positive and, because $g(m)=-\log(m)\to-\infty$ as $m\to\infty$, it follows from \eqref{apresFatou} that
\begin{equation}
\label{tendtoinfinity}
\liminf_{M \to \infty} \frac{1}{M} \sum_{m=1}^M \CP^{(n)}(\Test_{m,K}) = \infty.
\end{equation}
Recalling the upper bound \eqref{CPisbounded} on $\CP^{(n)}(\Test_{m,K})$, we finally reached a contradiction. The proof of Theorem~\ref{Theo-rigidity} is therefore complete. 

\subsubsection{A more general statement for number-rigidity}
\label{Sec:general}

A careful examination of the proof of Theorem~\ref{Theo-rigidity} reveals that we did not use much of the properties of the vector space $\R$ nor of the logarithmic interacting potential $g$. More precisely, the properties of $g$ are only used in equalities \eqref{echangepointslogm}, \eqref{secondestimatecost} and when stating that
$$
\liminf_{M \to \infty} \frac{1}{M} \sum_{m=1}^M \e^{-(C g(m) + c)}=\infty. 
$$
As for the compact sets $[0,1]$, $[1,2]$ and $[m,m+1]$ appearing in the proof,  they can be replaced by arbitrary disjoint compact sets $K_0,K_1,K_m$ with unit Lebesgue volume and such that the distance from $K_m$ to $K_0$ and $K_1$ goes to infinity with $m$. 

Thus, let us consider more general interaction potentials $g:\R^d\to\R\cup\{+\infty\}$ for $d\geq 1$ and redefine   \eqref{def:Hlambda}--\eqref{fKernel} accordingly by using this new $g$ in their definition (where we now set $\La_m:=[-\tfrac m2,\tfrac m2]^d)$.  Let us also assume there exists a non-trivial class $\mathscr C$ of stationary point processes  for which the move functions exists, namely so that the results of Lemma~\ref{lem:definimoveinf} holds for any $P\in\mathscr C$ and, having in mind the proof of Proposition~\ref{splitErgo}, that $\mathscr C$ is stable by desintegration: if $P\in\mathscr C$ can be written as \eqref{mixingP}, then $P_\eta\in\mathscr C$ for $P$-a.e. $\eta$ (recall that for the one dimensional logarithmic interaction  $\mathscr C$ were  the class of stationary point processes that have finite renormalized energy). We then say that a stationary point process $P$ on $\R^d$ satisfies the \emph{canonical DLR equations with respect to $g$} if $P\in\mathscr C$ and \ref{canDLR} holds with the new definition for $f_{\La,\R^d}$. In this more general setting, cosmetic modifications of the proof of Theorem~\ref{Theo-rigidity} leads to the following result.

\begin{theorem} 
\label{DLRrigidityGeneral}
Let $d\geq 1$ and $g:\R^d\to\R\cup\{+\infty\}$ be a measurable function which satisfies $g(x)\to-\infty$ when $\|x\|\to\infty$ and assume that there exists a compact set $K\subset\R^d$ such that $g$ is continuous on $\R^d\setminus K$.
If $P$ is a stationary point process on $\R^d$ which satisfies the canonical DLR equations with respect to $g$, then $P$ is number-rigid. 
\end{theorem}

In particular, this results applies to the logarithmic potential $g(x)=-\log\|x\|$ on $\R^d$ for any $d\geq 1$, including the $d=2$ Coulomb interaction. The question of identifying an appropriate class $\mathscr C$ for this setting will be investigated in another work.  Note that Theorem~\ref{DLRrigidityGeneral} does not cover, however, the Coulomb interaction $g(x)=\|x\|^{-(d-2)}$ in dimension $d\geq 3$ or, more generally, the Riesz interactions $g(x)=\|x\|^{-s}$ for any $s\in\R$.

\bibliographystyle{plainnat}

\end{document}